\documentclass[10pt,a4paper]{article}

\usepackage{amsmath,amssymb,amsfonts,amsthm,latexsym,amscd}
\usepackage[english]{babel}
\usepackage{graphics}
\usepackage{color}
\usepackage{framed}
\usepackage[mathscr]{eucal}   	

\usepackage{hyperref}
\usepackage[all,2cell]{xy} \UseAllTwocells 		% for diagrams + 2-cells 	
\newcommand{\black}{\color{black}}

\newcommand{\B}{\mathcal{B}}

\newcommand{\G}{\mathcal{G}}
\renewcommand{\H}{\mathcal{H}} 	%re!
\newcommand{\K}{\mathcal{K}} 
\newcommand{\M}{\mathcal{M}}
\newcommand{\U}{\mathcal{U}}
\newcommand{\V}{\mathcal{V}}

\newcommand{\Ss}{{\mathscr{S}}}

\newcommand{\bbc}{{\mathbb{C}}}

\DeclareSymbolFont{rsfs}{U}{rsfs}{m}{n}
\DeclareSymbolFontAlphabet{\scr}{rsfs}% was ... mathscr

\newcommand{\esspal}{G_{ep}}

\DeclareMathOperator{\InnAut}{InnAut_{\pi}}
\DeclareMathOperator{\SpatialAut}{SpatialAut_{\pi}}
\DeclareMathOperator{\Homeo}{Homeo}
\DeclareMathOperator{\Diff}{Diff}
\DeclareMathOperator{\Bis}{Bis}
\DeclareMathOperator{\Ad}{Ad}
\DeclareMathOperator{\ls}{clspan}

\DeclareMathOperator{\kernel}{ker}
\DeclareMathOperator{\rank}{rank}

\DeclareMathOperator{\Bisherm}{Bis_{\mathrm{herm}}}

\newtheorem{theorem}{Theorem}[section]				% theorems numbered under sections
\newtheorem{proposition}[theorem]{Proposition}

\newtheorem{lemma}[theorem]{Lemma}

\theoremstyle{definition}        
\newtheorem{definition}[theorem]{Definition}
\newtheorem{remark}[theorem]{Remark} 

 \newtheorem{example}[theorem]{Example} 
  \newtheorem{examples}[theorem]{Examples}
 \newtheorem{comment}[theorem]{Comment}

\title{C*-bundle dynamical systems}
\author{\normalsize  
Rachel A.D. Martins 
\\
\normalsize  \textit{Centro de An\'alise Matem\'atica e Sistemas Din\^amicos, Departamento de
Matem\'atica,}
\\
\normalsize \textit{Instituto Superior T\'ecnico, Universidade T\'ecnica de Lisboa,}
\\
\normalsize \textit{Av.~Rovisco Pais 1, 1049-001 Lisboa, Portugal}
 \thanks{Research supported by Funda\c{c}$\tilde{\mathrm{a}}$o   para
as Ci\^encias e a Tecnologia (FCT)
including programs POCI 2010/FEDER and SFRH/BPD/32331/2006. Email: rmartins-at-math.ist.utl.pt}}

\begin{document}

\maketitle

\begin{abstract}
C*-bundle dynamical systems are introduced and their r\^ole within the theory of
C*-subalgebras
and Fell bundles is investigated. A C*-bundle dynamical system involves an action of a 1-parameter
group of ``spatial
automorphisms'' of the C*-bundle together with a notion of covariance with respect to the
diffeomorphisms of the base manifold, and turns out to define a class of examples of Arveson's
$A$-dynamical systems. An embedding invariant for non-commutative C*-subalgebras (equivalent to a
groupoid 2-cocycle by construction) emerges and has presentation analogous to a partition function
towards a potential algebraic formulation of quantum gravity.
\end{abstract}

\section{Introduction}

Many physicists including Heisenberg, have highlighted the relevance of
non-commutative configuration C*-algebras $A$ as subalgebras in a C*-algebra of observables $B$, so it
is important to establish a good characterisation of the embeddings of not necessarily
abelian C*-algebras $A$ in ambient C*-algebras $B$. Fell bundles as groupoid extensions to C*-bundles
can provide a rich context for addressing problems in non-commutative geometry and physics. 

 \medskip
 
This paper is a byproduct of a program of research motivated by attempts at translating certain ideas
from spin foam quantum gravity into C*-algebraic language. Some of the new definitions emerged in part
thanks
to discussions with Paolo Bertozzini, Roberto Conti and Pedro Resende. There is no single main theorem,
instead we present 5 smaller puzzle pieces theorems, plus a proposition and 2 lemmas, which all serve
to address and illustrate the r\^ole of C*-bundle dynamical systems within the theory of C*-subalgebras
and Fell bundles over groupoids, always in the light of their applications to physics.

\medskip

We open with a review of Banach bundles, C*-bundles (or C*-algebra bundles), Fell bundles over
groupoids with examples for the theoretical physicist unfamiliar with these topics. We follow those
preliminary ideas with a study of a suitable notion of self-map for a C*-bundle $(E^0, \pi, X)$ over a
locally compact space $X$ and relate them to unitary representations of $\Homeo(X)$ on $E^0$. Next, we
show how those notions relate to Kumjian's normalisers for embeddings of $C^(E^0)$ as a C*-subalgebra in
$C^*(E)$. ($(E^0,\pi,X)$ is the restriction to $X = G^0$ of a Fell bundle $(E,\pi,G)$ over a groupoid
given by $G = X \times X$.)  For coherence and for the convenience of the reader we recall a few
background details about C*-subalgebra theory as developed by Kumjian, Renualt and Exel and discuss
generalisations to non-commutative C*-subalgebras.\footnote{These C*-subalgebras in C*-algebras are
commonly denoted by $B$ for the subalgebra and $A$ for the ambient algebra. We choose the opposite
notation convention here: $A \subset B$, because of the connections to be made later to other topics
also involving C*-subalgebras. We hope that this does not cause an inconvenience to readers familiar
with the $B \subset A$ convention.} Then we introduce C*-dynamical systems and give a precise
comparison with Arveson's $A$-dynamical systems (and the more general notion of C*-dynamical system)
discussing the advantages to physics of C*-bundle dynamical systems. We construct an ``embedding
invariant'' $\Phi_{\hookrightarrow}$ to be equivalent to a groupoid 2-cocycle $\omega$ and it has an
alternative presentation in a form which is analogous to a partition function for quantum gravity. We
presume that this is the first appearance of such a partition function towards an algebraic formulation
of quantum gravity. Finally, we show how a C*-bundle dynamical system can provide a bridge or
enveloping structure between $(E,\pi,G)$ for a principal groupoid $G$ and $(A,B)$ for $A$ a diagonal
C*-subalgebra \cite{C*diag} in $B$.

\medskip

Additional potential applications to physics may include: (i) The (strongly continuous) action
taking part in a C*-bundle dynamical system is a lifting of a 1-parameter group of diffeomorphisms of
the base differential manifold $M$. Moreover, the construction is similar to ideas about modular
quantum gravity first appearing in \cite{BCL mst}. (ii) The covariance condition that applies to a
C*-bundle dynamical system may have an application to the Poincar\'e covariance axiom from
algebraic quantum field theory. (iii) Although in this paper we only work with the reversible
case, irreversible C*-bundle dynamical systems should have applications to decoherence theories,
entanglement and to topological inverse semigroups such as those describing Penrose tilings. (iv) The
generator of the 1-parameter group of C*-bundle spatial automorphisms (over a simply connected compact
manifold) should lead to an interpretation of a Dirac operator in a spectral triple as a
generalised connection on a C*-bundle. There should also be a clear relation between the
$A$-bimodules of 1-forms $\Omega_D^1(A)$ in a finite spectral triple $(A,\H,D)$, and both Arveson paths
and Exel slices. We leave these geometrical considerations for further work.

\section{Fell bundles and C*-bundles}

In this section we aim to include a comprehensive set of preliminary definitions for completeness
and accessibility, above all to aid physicists who may be unfamiliar with these topics
but are interested in the applications of the results presented below. 

\medskip

For Dixmier, Fell and many others, Banach bundles and C*-bundles are not equivalent to vector bundles 
with additional structure:

\begin{definition}\cite{Fell Doran}
A \emph{Banach bundle} ($E, \pi, X$) is a surjective continuous open map $\pi: E \rightarrow X$ 
such that
$\forall x \in X$ the fibre 
%+ ##paolo 07-09-2011## 
% added the notation \E_x that we need to use later on 
$E_x:=\pi^{-1}(x)$ 
%-
is a complex Banach space and satisfies the following additional conditions:

\begin{itemize}
\item 
the 
%+ ##paolo 07-09-2011## 
% changed: operator -> operation 
operation 
%- 
of addition $+ : E \times E \rightarrow E$ is continuous on the set 

%+ ##paolo 07-09-2011## 
% added spaces before/after | 
$E \times_{X} E := \{ (e_1,e_2) \in E \times E \ | \ \pi(e_1) = \pi(e_2) \}$,
%-
\item 
the operation of multiplication by scalars: $\bbc \times E \rightarrow E$ is continuous,
\item 
the norm 
%+ ##paolo 07-09-2011## 
% changed: . -> \cdot 
$\parallel \cdot \parallel : E \rightarrow \mathbb{R}$ is continuous,
%- 
\item 
for all $x_0 \in X$, the family $U_{x_0}^{\mathcal{O},\epsilon} = \{e \in E \ | \ \parallel e \parallel
<
\epsilon, \pi(e) \in \mathcal{O} \}$ where 
$\mathcal{O} \subset X$ is an open set containing $x_0 \in X$ and $\epsilon > 0$ is a fundamental
system of
neighbourhoods of 
%+ ##paolo 07-09-2011## 
% changed: \mathcal{O} -> 0 
% this is the zero element of the vector space \E_x 
$0 \in E_{x_0}$.    
%- 
\end{itemize}
For a \emph{Hilbert bundle} we require that for all $x \in X$, the fibre $E_x$ is a Hilbert space. 
\end{definition}

Modulo our notation, the current paragraph of remarks is taken from \cite{Fell Doran}. Let us say that a
Banach bundle $E$ is locally trivial if for every $x \in X$ there is a
neighbourhood $U$ of $x$ such that the restricted Banach bundle $E_U$ is isomorphic to a trivial
bundle.
Of course not all Banach bundles are locally trivial. It can be shown that a Banach bundle over a
locally compact Hausdorff space, whose fibres are all of the same finite dimension, is necessarily
locally trivial.

\begin{definition}
 A \emph{morphism of Banach bundles} $(f,f_0) : (E,\pi,X) \to (E',\pi',X)$ 
 %whose base spaces are identified 
%(or should they be only identified up to a homeomorphism $f_0$? In which case the covariance
%condition $\pi = \pi' \circ f$ is replaced with $f_0 \circ \pi = \pi' \circ f$ as above) 
consists of a (norm-decreasing) continuous linear map $f:E \to E'$ and a continuous linear map $f_0:X
\to X$ such that

\begin{itemize}
 \item 
$\pi' \circ f =  f_0 \circ \pi$;
 \item
 each induced fibrewise map $f_x : E_x \to E'_{f_0(x)}$ is continuous.
 %\item
%whenever $\parallel s_i \parallel \to 0$ and $\pi(s_i) \to x$ in $E$, then $f(s_i) \to 0_x$ in $E'$.
%Equivalently, $E'_U \cap T_{\epsilon'} \subset V'$ with $T_{\epsilon'} = \{e' \in E' : \parallel e'
%\parallel < \epsilon'$, $V'$ open in $E'$ and containing $0_x$ (the zero of $E'_x$).
\end{itemize}
If $f$ and $f_0$ are invertible and $\parallel e \parallel_{E} ~=~ \parallel f(e) \parallel_{E'}$,
then $(f,f_0)$ define a
\emph{Banach bundle isometric isomorphism}.
\end{definition}
For equivalent definitions of this isometric isomorphism we refer to the books \cite{Fell Doran} and
\cite{Dixmier}. Also in \cite{Dixmier}, Dixmier gives an alternative definition of Banach bundle as a
continuous field of Banach spaces over a topological space.

\begin{definition}\cite{Fell Doran}
 Let $X$ be a locally compact Hausdorff space. By a bundle of C*-algebras (or \emph{C*-bundle}) we mean
a Banach bundle $(E^0,\pi,X)$ together with a multiplication $\cdot$ and involution * in each
fibre $E^0_x$ of $E^0$, such that

\begin{enumerate}
 \item For each $x \in X$, $E^0_x$ is a C*-algebra under the linear operations and norms of $E^0$ and
the operations $\cdot$ and *;
 \item the multiplication is continuous on 
 
 \begin{equation}
  \{  <a,b> ~ \in E^0 \times E^0 : \pi(a) = \pi(b) \} 
 \end{equation}
 to $E^0$; and
 
 \item the involution * is continuous on $E^0$ to $E^0$.
\end{enumerate}
\end{definition}

We associate a C*-algebra $A$ to a C*-bundle $E^0$ as follows.
The algebra of compactly supported continuous sections $C_c(E^0)$ completed in the operator norm is a
concrete C*-algebra (operating on $\H = L^2(E^0)$, the inner product norm completion of $C_c(E^0)$),
which we call the enveloping algebra $A$ of $E^0$. \black We also use the symbol $C^*(E^0)$ when it is
helpful
to emphasise that $A$ arises from a C*-bundle $E^0$ in this way. 

\medskip

(A Fell bundle over a groupoid is a generalisation of a Fell
bundle over a topological
group \cite{Fell Doran}, and also a generalisation of a $C^*$-bundle),

\begin{definition} \cite{fbg}   \label{defining list}
A Banach bundle over a groupoid $p: E\rightarrow G$ is said to be a \textit{Fell bundle} if there
is a continuous multiplication $E^2 \rightarrow E$, where
\begin{displaymath}
 E^2 = \{(e_1,e_2)\in E \times E \ | \  (p(e_1),p(e_2)) \in G^2\},
\end{displaymath}
(where $G^2$ denotes the space of composable pairs of elements of $G$) and an involution $e
\mapsto
e^{\ast}$ that satisfy the following axioms. 
%+ ##paolo 07-09-2011## 
%($\E_{\gamma}$ is the fibre $p^{-1}(\gamma)$).
%-

\begin{enumerate}
\item 
$p(e_1e_2) = p(e_1)p(e_2) \quad \forall (e_1,e_2) \in E^2$;
\item 
The induced map $E_{g_1} \times E_{g_2} \rightarrow  E_{g_1 g_2}$,\quad 
$(e_1,e_2) \mapsto e_1e_2$ is bilinear $\forall (g_1, g_2)\in G^2$;  \item
$(e_1e_2)e_3=e_1(e_2e_3)$ whenever the
multiplication is defined; 
\item $\parallel e_1e_2 \parallel \leq \parallel e_1 \parallel \cdot \parallel e_2 \parallel, \quad
\forall
(e_1,e_2) \in E^2$;
\item $p(e^{\ast})=p(e)^{\ast},\quad \forall e \in E$;
\item The induced map $E_{g} \rightarrow E_{g^{\ast}}, \quad e \mapsto e^{\ast}$ is conjugate
linear for all $g \in G$;
\item $e^{\ast \ast} = e, \quad \forall e \in  E$;
\item $(e_1e_2)^{\ast} = e_2^{\ast} e_1^{\ast},\quad \forall (e_1,e_2) \in E^2$;
\item $\parallel e^{\ast} e\parallel =\parallel e \parallel ^2, \quad \forall e \in E$;
\item $\forall e \in E, \ e^{\ast} e \geq 0$ as element of the C*-algebra $E_{p(e^*e)}$. 
\end{enumerate}
\end{definition}
The information in the following two paragraphs is also recalled from \cite{fbg}.
The fibres $\{ E_x \}_{x \in G_0}$ over the unit space $G_0$ of $G$ are C*-algebras. A \textit{unital
Fell bundle} is one in which each of these $C^*$-algebras has an identity element. A Fell bundle is
said to be 
\textit{saturated} if $E_{g_1 g_2}$ is the closed linear span of $E_{g_1}.E_{g_2}$ for all $(g_1,
g_2)~\in~G^2$. In this case, $E_{g} \otimes E_{g^*} \cong E_{g g^*}$ and  $E_{g^*} \otimes E_{g}
\cong E_{g^* g}$ for all $g, g^* \in G$, (including $gg^*,g^*g \in G_0$) and the fibres of $E$ are
called imprimitivity bimodules or Morita-Rieffel equivalence bimodules, or in other words, the
C*-algebras $E_{g g^*}$ and $E_{g^* g}$ are \emph{strongly Morita equivalent} or \emph{Morita-Rieffel
equivalent}. All Fell line bundles over groupoids (that is, a Fell bundle $(E, \pi, G)$ with fibre
$\bbc$) are saturated since in one dimension, two algebras are Morita equivalent exactly when they are
isomorphic.

\medskip

The \emph{enveloping algebra}
$C_r^*(E)$ of a Fell line bundle $(E, \pi,G)$ is the algebra of
compactly
supported continuous sections $C_c(E)$ of $(E,\pi,G)$, completed in the operator norm: $C_r^*(E) \subset
\B(L^2(E))$,
where $L^2(E)$ is the Hilbert space obtained from the completion of $C_c(E)$ in the inner product norm.
We also denote the enveloping algebra of $E$ as $C^*(E)$. Consider the restriction of the
image of the
surjection $\pi:E \to G$, to $G_0$, the object space or unit space of $G$, and let $(E^0,\pi,G_0)$
denote
the C*-bundle 
corresponding to this restriction. Put $P : C^*(E) \to C^*(E^0)$ for the restriction of the enveloping
algebras. $A = C^*(E^0)$ is also referred to as the \emph{diagonal algebra} of the Fell bundle.

\medskip

%Recall that the pure state space $X(A)$ of a C*-algebra $A$ is identified with
%the unitary equivalence classes of irreducible representations of $A$. Throughout this paper all
%examples of C*-algebras are assumed to be separable so that there is a
%sequence of irreducible representations $\{ \pi_n \}$ of $A$ such that $\bigoplus_n \pi_n$ is a
%faithful representation of $A$. Since we are following \cite{Renault}, we will also assume that $X$ is
%a
%locally compact Hausdorff space. When valid and appropriate, it will also be assumed that $X$ is a
%topological manifold $M$ admitting a smooth structure. A well known reference on C*-algebras is
%\cite{Dixmier}.

Here are some examples of C*-bundles (or C*-algebra bundles), followed by examples of Fell
bundles:-

\begin{examples} \label{Ex C bundles}
 \begin{enumerate}
     \item  Recall that every arbitrary finite dimensional C*-algebra takes the form
$A=\bigoplus_{i=1}^m
M_{n_i}$ up to a canonical isomorphism. $A$ is the enveloping algebra of a C*-bundle $E^0$ where $X$ is
a discrete space with $m$ points $\{ x \}$ indexed by $i$ and each
fibre $E^0_{ \{ x \} }$ of $E^0$ is given by a simple matrix algebra $M_{n_i}$.
  %\item  Recall that an approximately finite dimensional (AF) C*-algebra is the direct (or inductive)
%limit of a sequence of finite dimensional C*-algebras $A = \overline{\cup_n A_n}$. Let $A$ be the
%enveloping algebra of a C*-bundle $E^0$ over a discrete space $X$.
%The projections in $\cup_n A_n$ form an approximate unit for $A$.
 % \item Let $E$ be a Fell bundle over a topological group $\G$ (or C*-algebraic bundle over
%$\G$). Then $E$ is a C*-bundle over $X=\G$.
 \item A fundamental example arises from treating $C_0(X)$ (the algebra of
continuous functions vanishing at infinity on a locally compact Hausdorff space $X$) as
the algebra of
continuous sections (vanishing at infinity) of a complex line bundle over $X$. 
 \item A tensor product of the C*-algebra $C_0(X)$ with a C*-algebra $E^0_x$. (For example, a minimal
C*-tensor product of C*-algebras $C$ and $D$ acting on Hilbert spaces $H$ and $K$, was defined by
Tomiyama in \cite{TomiyamaTensorProducts} to be the closure of the algebraic tensor product $C \otimes
D$ by the operator norm of the C*-algebra of bounded linear operators on $H \otimes K$.) 
\item  A continous field of elementary C*-algebras, usually satisfying Fell's condition.
(This is how the continuous trace-class C*-algebras arise.)
\item A Banach bundle whose fibre is a complex simple C*-algebra isomorphic to a Clifford algebra.
\item Let a C*-bundle $(E^0,\pi,X)$ be constructed from a separable C*-algebra $A$ as
follows. Since $A$ is separable, we identify $A$ with $\bigoplus^X_m \pi_m$, the direct sum over
all irreducible representations $\pi_m$ of $A$. Define the space of fibres of $E^0$ to be given by $\{
E^0_x \}_{x \in X} = \{ \pi_m \}$, where $X$ is identified with the pure state space $X(A)$, which is
locally compact in the weak *-topology. Then the enveloping algebra $C^*(E^0)$ of $E^0$ is identified
with $A$ and operates on $\H = L^2(E^0)$. 
 \end{enumerate}
\end{examples}

%Fell bundles examples

From \cite{fbg}, any saturated Fell bundle is equivalent to a semidirect product arising from
the action of a locally compact groupoid on a C*-bundle as follows.

\begin{example}[Semidirect product Fell bundle] \label{FB example 1}
Following \cite{fbg} one constructs a saturated Fell bundle $E$ over a topological
groupoid $(E,\pi,G)$ from a semidirect product structure as follows. Let $E^0$ be a C*-bundle over
$G_0$ and let $r$ and $d$ denote the domain and range maps of $G$. Let the product of
elements $e_1 =  (g,a) \in E$ and $e_2 =(h,b) \in E$, (with $a,b \in E^0$ such that $\pi(a)=d(g)$,
$\pi(b)=d(h)$), for each pair $(g,h)$ such that $r(g)=d(h)$, be given by: 
 \begin{equation*}
  e_1e_2 = (gh, \alpha_g(a)b)
 \end{equation*}
where $\alpha_g$ is an isometric *-isomorphism of fibres $\alpha_g: E^0_{d(g)} \to E^0_{r(g)}$ defined
by
$\alpha_g(a) = u a u^*$ with unitary elements $u \in E_g$, $u^* \in E_{g^*}$. The involution
on $E$ is given by $e_1^* = (g,a)^* = (g^*, \alpha_g(a^*))$. 
We denote the resulting Fell bundle by $E = G \ltimes E^0$.
This semidirect product structure induces an action of $G$ on the enveloping algebra $A = C^*(E^0)$ of
the C*-bundle $E^0$, such that $C^*(E) = G \ltimes C^*(E^0)$. 

\medskip

(a) In the case that $E$ is a Fell line bundle over a locally compact
\'etale groupoid
$G$, then $C_r^*(E)$ is identified with the twisted convolution C*-algebra of $G$ (see for
example \cite{Renault}) and for the trivial action $\alpha=1$, then $C_r^*(E)$ is identified with the
ordinary or untwisted group convolution algebra $C^*_r(G)$ of $G$.

(b) Let $E$ be a locally trivial Fell bundle with non-commutative fibre over a pair groupoid over a
locally compact simply connected (possibly discrete) manifold $G=M \times M$ such that the fibre of the
corresponding C*-bundle $(E^0, \pi, M)$ is a simple C*-algebra.
\end{example}

\begin{example}[Imprimitivity Fell bundle] \label{Example discrete FB}
Let $E$ be a unital saturated Fell bundle over a discrete groupoid $G$ whose induced C*-bundle over
the discrete space $G_0$ is given by $E^0=\bigoplus M_{n_i}(\bbc)$. (Let $G_0$ have $i$ points and the
fibres of $E^0$ be simple matrix algebras of varying dimension $n$.) The fibres of $E$ are
$M_{n_i}(\bbc)$-$M_{m_j}(\bbc)$ imprimitivity or Morita equivalence bimodules. In other notation, the
fibres $E_g$ are given by imprimitivity $E_{d(g)}$-$E_{r(g)}$-bimodules. Since $G_0$ is discrete, $E$ is
locally trivial as Banach bundle, although its fibres are not in general topologically equivalent. In
the case that $G=X \times X$, then $E$ defines a full C*-category \cite{BCL Imp}. For C*-categories see
\cite{GLR}.
\end{example}

%\% Take out regular stuff and put it in the second paper.
%Note that even though the
%algebra $C^*(E^0)$ is regular in $C^*(E)$, $C^*(E)$ is not a crossed product algebra arising from the
%action of a groupoid (or group).
%\begin{example}
%An example of a Fell bundle that is not
%saturated is
%given by a Fell bundle $E$ over inverse semigroup $S$ if the multiplication of sections is
%given by $(s_1,a_1)(s_2,a_2)=(s_1s_2, \alpha_{s_1}(a_1)a_2)$ then $\alpha_{s_1}$ does not induce an
%isomorphism of fibres $E_1 \to E_2$. This is an example of a regular Fell bundle ($C^*(E^0)$ is a
%regular C*-subalgebra of $C^*(E)$ since for $b \in C^*(E)$ and for all $a \in C^*(E^0)$, and the set of
%all $b$ such that $aba^* \subset C^*(E^0)$, generates $C^*(E)$).
%\end{example}

\section{Spatial automorphisms and ambient Fell bundles} \label{spatial section}

\subsection{Spatial automorphisms}

An automorphism of a C*-bundle with enveloping algebra $A$ consists of
a Banach bundle isometric isomorphism preserving the structure of the bundle, extending to an isometric
*-isomorphism $\alpha: A \to A$. In examples there may exist automorphisms that are not implemented by
unitary operators, that is, that are not spatial, but we will only make use of spatial automorphisms.

\begin{definition} \label{IA}
 Let $(E^0,\pi,X)$ be a C*-bundle over a locally compact space $X$ with enveloping
C*-algebra
$A = C^*(E^0)$ represented on a separable Hilbert space $\H$. A \emph{spatial automorphism of the
C*-bundle}
consists of invertible continuous linear maps $f_0$ and $f$ with commuting diagram:
 
 \begin{equation} 
  \xymatrix{
        E^0 \ar[r]^f \ar[d]_{\pi} & E^0 \ar[d]^{\pi} \\
        X \ar[r]_{f_0} & X  }. 
 \end{equation}
 such that each induced fibrewise linear map $f_x : E^0_x \to E^0_{f_0(x)}$ is invertible 
continuous 
 and such that $f$ extends to an isometric *-isomorphism $\hat{f} : A \to A$ of the form  $\hat{f}(a) =
UaU^*$ where $U$ is a unitary linear map on $\H$. 
 \end{definition}
 The previous definition can be thought of as a special case of a Fell bundle morphism from \cite{BCL
En}.
 %We will assume that $X$ is a compact space and $A$ has a unit $1_A$.
 
 \medskip

Suppressing the circumflex, we will denote C*-bundle spatial automorphisms by pairs $(f_0,f)$ or
sometimes more explicitly, by triples $(f_0,f,U_{f_0})$. We will refer to the bundle structure
preservation condition 

\begin{equation} \label{covariance}
 \pi \circ f = f_0 \circ \pi
\end{equation}
as the \emph{covariance
condition}\footnote{Observe that $(f_0,f)$ is permuting and transforming the fibres. In finite
dimensions, $U$ is given by a matrix with one non-zero block in each row and each column of blocks.}.

\medskip

In our related work, we will see that \eqref{covariance} is one of the main features distinguishing
C*-bundle dynamical systems from other C*-dynamical systems.

\medskip

Clearly, these ``covariant'' automorphisms $(f_0,f)$ define a subgroup
$\SpatialAut(E^0)$ of the group
$\mathrm{Aut}(A)$ of all automorphisms of the C*-algebra $A$.

%(These covariant inner automorphisms
%are not to be confused with covariant representations of $A$.)
  
   %and the inner automorphisms (suppressing the circumflex) $f$ form
%a group $\InnAut(E^0)$.

\medskip

In the case that the fibres of $E^0$ are of varying dimension or are topologically inequivalent i.e.
if $E^0$ is not locally trivial, then if $(f_0,f)$ is a spatial automorphism of
$E^0$, then $f_0$ is the identity homeomorphism.

\begin{remark}
\textit{It was kindly pointed out by Roberto Conti
that if $A$ is a C*-algebra $\K(\H)$ consisting of all compact operators on a Hilbert space $\H$
(considering the case
of a C*-bundle with only one fibre, given by $\K(\H)$), then all automorphisms $\phi$ of $A$ are
implemented by unitaries operators $U$ on $\H$ such that $U \notin A$, so that in this case $\phi$ is
never inner.}
\end{remark}

For each C*-bundle $(E^0,\pi,X)$ over a locally compact Hausdorff space, the homeomorphisms $f_0 : X \to
X$ form a group, $\Homeo(X)$. 

\medskip

We emphasise that the group
of global bisections $\Bis(G)$ of the pair groupoid $G=X \times X$ over a topological space $X$ is
identified with the group $\Homeo(X)$ of homeomorphisms of $X$, and if $M$ is a smooth manifold,
then $\Bis(G)$ for $G = M \times M$, is identified with the group $\Diff(M)$ of diffeomorphisms. Let
$\Bisherm(G)$ denote the abelian subgroup of self-adjoint diffeomorphisms, 
\begin{equation} \label{Bisherm}
 g=g^* \in \Bis_{\mathrm{herm}}(G) \subset \Bis(G).
\end{equation}

Recall that the general linear groupoid $GL(V)$ of a vector bundle $V$ over a space $X$ is the canonical
groupoid of linear
isomorphisms between the fibres of $V$ and that a representation $\rho_G$ of a groupoid $G$ on a vector
bundle $V$ is a groupoid homomorphism
$\rho_G: G \to GL(V)$ into the general linear groupoid of $V$.  

\begin{definition}
The \emph{general linear groupoid} $GL(E^0)$ of a C*-bundle $(E^0, \pi, X)$ with isomorphic fibres, is
the set of all isometric *-
isomorphisms $\phi_{(x,y)}$ between each pair of C*-algebras:
\begin{equation}
 GL(E^0) = \{ \phi_{(x,y)}: E^0_x \to E^0_y ~~\vert~~ x,y \in X \},
\end{equation}
together with the canonical composition of isomorphisms,
inverses and units $\iota_{x}:E^0_x \to E^0_x$. 
\end{definition}

\begin{lemma} \label{lemma}
 Let $(E^0, \pi, X)$ be a C*-bundle. $\SpatialAut(E^0)$ is a subgroup of the group $\Bis(GL(E^0))$
of global bisections of the groupoid $GL(E^0)$.
\end{lemma}
\begin{proof}
 The global bisections $\alpha \in \Bis(GL(E^0))$ of $GL(E^0)$ satisfy the covariance condition $\pi
\circ \alpha = f_0 \circ \pi$ where $f_0$ is a bisection of the groupoid $G=X \times X$. 
\end{proof}

%- intuitively,
%because of the way it permutes and transforms fibres. 

%(If $E^0$ is locally trivial, we may
%often formally identify $GL(E^0)$ with the groupoid of isomorphisms on the underlying vector bundle
%$E^0$, well known as the general linear groupoid of the vector bundle.)

%\begin{definition}
% A \emph{representation} $\rho$ of $\Homeo(X)$ on $E^0$ is a group homomorphism, 
% \begin{equation}
%  \rho : \Homeo(X) \to \Bis(GL(E^0)).
% \end{equation}
%\end{definition}

Let $\U(\H)$ denote the group of unitary
linear maps on a Hilbert space $\H$ and let $\psi \in \H$. Recall that a strongly continuous
unitary representation of a group $\G$ is given by a group homomorphism $g \mapsto U_g \in \U(\H)$ such
that $g \mapsto U_g \psi$ defines a continuous map from $\G$ to $\H$.

\begin{definition} \label{unitary rep}
 Let $f_0,g \in \Homeo(X)$. A \emph{unitary representation} $\rho_U$ of $\Homeo(X)$ on $E^0$ is
a group homomorphism, 
 \begin{gather}
  \rho_U : \Homeo(X) \to \SpatialAut(E^0),\\ 
g \mapsto U_g a U_g^*,\\
\textrm{or} ~~ f_0 \mapsto (f_0,f,U_{f_0})
 \end{gather}
  \end{definition}
A unitary representation $\rho_U(\Homeo(X))$ is said to be \emph{strongly continuous} if the map from
$\Homeo(X)$ to $\H = L^2(E^0)$ defined by
\begin{gather}
 g \mapsto \alpha_g(a)\psi, ~~\psi \in \H,
\end{gather}
is continuous for each $a \in A$. 

\medskip

Of course, one may extend the previous definition to unitary representations of subgroups $\G \subset
\Homeo(X)$ and moreover, if $M$ is a smooth manifold and $\G$ is a subgroup of $\Diff(M)$, one defines a
unitary representation $\rho_U$ of $\G$ on a C*-bundle $(E^0,\pi,M)$ as a group homomorphism
$\rho_U : \G \to \SpatialAut(E^0)$. 

\medskip

If $\rho_U(\G) = \Bis(\rho_G)$ for some subgroup $\G$ of $\Homeo(X)$ or $\Diff(M)$, where $\rho_G : G
\to GL(E^0)$ is a groupoid homomorphism, then we say that $\rho_U$ induces a groupoid representation
$\rho_G$ on the C*-bundle $E^0$.

\medskip

Note that each unitary representation $\rho_U$ on $E^0$ of a group $\G$ provides an action $\alpha$ of
the group $\G$ on the C*-algebra $A$. 

See \cite{Rogier} for unitary representations of groupoids on Hilbert bundles.

\subsection{Covariance group}

Let $M$ be a locally compact Hausdorff space admitting a differential structure.

 \begin{definition}\label{covariance group}
 Let $(E^0,\pi,M)$ be a C*-bundle over a differentiable manifold $M$ and let
$\{ g(\lambda) \}_{\lambda \in \mathbb{R}}$ be a 1-parameter subgroup of $\Diff(M)$. A \emph{strongly
continuous 1-parameter covariance group} $\G_{\sigma} \subset \SpatialAut(E^0)$
is the image of a strongly continuous representation of $\{ g(\lambda) \}_{\lambda \in \mathbb{R}}$ on
$E^0$. %implemented
%by unitary operators $U_{g(\lambda)}$ of the form $U_{g(\lambda)} = e^{i \sigma \lambda}$ for a choice
%of (possibly unbounded) self-adjoint operator $\sigma$ on $\H$.
 \end{definition}

Let $(E^0,\pi,M)$ be a C*-bundle as above in \ref{covariance group}. Each $\G_{\sigma}$ is a subgroup of
a unitary representation $\rho_U$ of $\Diff(M)$ such that $\rho_U = \Bis(\rho_G)$. If $\{ g(\lambda)
\}_{\lambda \in \mathbb{R}}$ has a minimal flow (having a dense orbit) on $M$, we say that $\G_{\sigma}$
has a minimal flow across the fibres of $E^0$. The significance of the minimal flow for this paper is
that $\{ y ~:~ g(x) = y,~~ x,y \in M, ~~ g \in \{ g(\lambda) \}_{\lambda \in \mathbb{R}}  \}$ is dense
in $M \times M$.

\medskip

We will leave details on the infinitesimal generator $\sigma$ of $\G_{\sigma}$ for another paper.
Applications are expected to arise in the context of modular quantum gravity \cite{BCL mst}.

\subsection{The ambient Fell bundle}

Crossed product C*-algebras arise in C*-dynamical systems and are by definition regular (see below).
Kumjian observed that the structure of these C*-algebras can be illuminated by studying the way in
which C*-subalgebras embed in them, using the set of normalisers: 

\begin{definition}\cite{C*diag} \label{N}
 Suppose that $A$ is a C*-subalgebra of a C*-algebra $B$. An element $b \in B$ is said to
\emph{normalise} $A$ if 
 \begin{itemize}
  \item $b^* A b \subset A$ 
  \item $b A b^* \subset A$ 
 \end{itemize}
The collection of all such \emph{normalisers} is denoted $N(A)$. Evidently, $A \subset N(A)$;
further, $N(A)$ is closed under multiplication and taking adjoints. A normaliser, $b \in N(A)$ is said
to be \emph{free} if $a^2=0$. The collection of free normalisers is denoted $N_f(A)$.
\end{definition}

%\begin{example}\cite{C*diag}

%\end{example}

A C*-subalgebra $A$ in $B$ is said to be \emph{regular} in $B$ if the normalising set $N(A)$ in $B$
generates $B$, that is,  $\ls N(A) = B$, where $\ls$ denotes closed linear span.

\begin{proposition} \label{regular}
 Let $(E, \pi,G)$ be a unital Fell bundle from which we obtain C*-algebras $A \subset B$ such that $A
= C^*(E^0)$ and $B = C^*(E)$. If $E$ is saturated then $A$ is regular in $B$.
\end{proposition}
\begin{proof}
 For each $g,g^* \in G$, $E_g$ is a bimodule over $E_{gg^*}$ and $E_g \otimes E_{g^*}
\subset E_{gg^*}$. We may define a map $E_{d(g)} \to E_{d(g)}$ by $\alpha_b(a) = b a b^*$ for $b \in
E_g$ (observe that $b$ determines a free normaliser).
Since $E$ is saturated, $E_g \otimes E_{g^*} \cong E_{gg^*}$, therefore $b$ can chosen so that for any
$a' \in E_{d(g)}$ there is an $a \in A$ such that $\alpha_b(a) = a'$. We can repeat this process for
all fibres $E_g$ of $E$, and since any $n \in N(A)$ is a sum of free normalisers, it follows that $\ls
N(A) = B$.
\end{proof}

\begin{examples} \label{Examples regular}
\begin{enumerate}
 \item All complex line bundles over groupoids are saturated since in one dimension, two algebras are
Morita equivalent exactly when they are isomorphic.
 \item From \cite{fbg}, any saturated Fell bundle is equivalent to a semi-direct product arising from
the action of a locally compact groupoid on a C*-bundle. %(This results does not always apply in finite
%dimensions.)
 \item A Fell bundle with the structure of a crossed product of the form $C^*(E) = \G \ltimes A$
for a locally compact or discrete group $\G$ is saturated. A Fell bundle over a topological group is
called an algebraic bundle.
 \item Let $E^0=\bigoplus M_{n_i}(\bbc)$ over a discrete space of $i$ points so that the fibres of $E^0$
are the simple matrix algebra summands of varying dimension $n$. The fibres of $E$ are
$M_{n_i}(\bbc)-M_{m_j}(\bbc)$ imprimitivity or Morita equivalence bimodules. Note that even though the
algebra $C^*(E^0)$ is regular in $C^*(E)$, $C^*(E)$ is not a crossed product algebra arising from the
action of a groupoid (or group).
 \item Not all regular Fell bundles are saturated. An example of a Fell bundle that is not saturated is
given by a Fell bundle $E$ over inverse semigroup $S$ if the multiplication of sections is given by
$(s_1,a_1)(s_2,a_2)=(s_1s_2, \alpha_{s_1}(a_1)a_2)$ then $\alpha_{s_1}$ does not induce an isomorphism
of fibres $E_1 \to E_2$.
\end{enumerate}
\end{examples}

The following result demonstrates the relationship between a C*-bundle spatial automorphism group
$\SpatialAut(E^0)$ and the unitary normalisers of $C^*(E^0)$ in $C^*(E)$ that arises due to the
covariance condition \eqref{covariance}. (See also comment \ref{embed}). 

\medskip

Note that the subgroup of spatial automorphisms $(f,f_0)$ implemented by unitary elements $U$
of the ambient algebra $C^*(E)$, form a subgroup $S$ of the group of inner automorphisms
$\mathrm{InnAut}(C^*(E))$ of $C^*(E)$.
%This illustrates the point of view that a non-commutative embedding is an
%inner automorphism of a C*-bundle. 

\begin{theorem} \label{N(A)=InnAut}
Let $(E^0,\pi,X)$ be a C*-bundle with enveloping C*-algebra $A$ and let $(E,\pi,G)$ be a
saturated
Fell bundle over the pair groupoid $G$ over a locally compact space $X$ such that $E^0$ is the
restriction of $E$ over
$G_0$ and such that all fibres of $E^0$ are of equal dimension. Let $\Bis(G)$
denote the group of global bisections of $G$ and let $B$ denote the enveloping C*-algebra $C^*(E)$ of
$E$. The following groups are isomorphic:

\begin{enumerate}
 \item the group $S := \SpatialAut(E^0) \cap \mathrm{InnAut}(C^*(E))$ 
 %of spatial automorphisms of the
%C*-bundle $E^0$ which are also inner automorphisms of $C^*(E)$, (that is $(f,f_0)$ with $U \in
%C^*(E)$).
 \item the group of unitary normalisers $N_U(A)$ of $A$ in $B$. 
\end{enumerate}
\end{theorem}

\begin{proof}
 Each $U \in N_U(A)$ implements an element of $\Bis(GL(E^0))$ and gives an isometric *-isomorphism $f :
A \to A$ as follows. For $U \in N(A) \subset B$, we have $U a U^* \in A$. If $U$ were a section not
supported on a bisection $g \in \Bis(G)$ then $U a U^* \notin A$. Therefore $U$ is supported on $g \in
\Bis(G)$. Observe that if $U \in B$ implements some map
$f(a):=UaU^* \in B ~~\forall a \in A$, then the condition $\pi(U)=g=f_0$ for some $g \in \Bis(G)$ is
equivalent to the condition on $(f,f_0)$ that $\pi \circ f = g \circ \pi$, (\ref{lemma}). This condition
is
satisfied for $U \in N_U(A)$. Since $f$ defines an isometric *-isomorphism $A \to A$, then $(f,f_0) \in
\SpatialAut(E^0)$ and so $(f,f_0) \in S$. 

Conversely, every $(f,f_0) \in S$ can be implemented by some unitary normaliser $U \in
N_U(A)$, because the homeomorphism $f_0$ involves a permutation
transformation of the fibre set of $E^0$. In other words each $g \in
\Bis(G)$ designates a full permutation of the fibre set of $E^0$. Fix a $g \in \Bis(G)$ and let
Perm be
the corresponding set consisting of pairs of fibres $(E^0_x,E^0_y)$, where $x$ may be equal to $y$. 
Since $E$ is saturated,  $A=C^*(E^0)$ is regular in $B=C^*(E)$ (\ref{regular}). For each pair
$(E^0_x,E^0_y)$ in Perm, the free normalisers can provide isometric *-isomorphisms of the form, 

\begin{equation}
 \alpha_{xy}:E^0_x \to E^0_y, ~~ ~ ~ \alpha_{xy}(a_x) = u_{xy} a_x u_{xy}^* ~~~= a_y \in E^0_y.
\end{equation}%where $u_{yx}=u_{xy}^*.
with $u_{xy}$ a unitary element of the fibre $E_{xy}$ of $E$ and $u_{xy}^* \in E_{yx}$. Note
that each
$U \in N_U(A)$ is a full rank operator on $\H$. A unitary normaliser $U \in N_U(A) \subset \ls
N(A)$ can be approximated from a set of finite linear combinations of normalisers in $N(A)$. In this
way arises a $U_g \in N_U(A)$ and a map $f : E^0 \to E^0$, defined by $f(a) = U_g a U_g$,
where each finite linear combination is a sum of maps $\alpha_{xy}$ parametrised by elements of the set
Perm. This results in an assignment $(f_0 = )~ g \mapsto U_g \in N_U(A)$. 
\end{proof}

\section{Non-abelian C*-subalgebras} \label{non-ab}

In view of the previous two sections, one might ask if C*-bundles and their spatial automorphisms ought
to be included or embedded in a larger algebraic context, seeing as the set $N_U(A) \notin A$. (See
later section \ref{Embeddings section} on embedding invariants.) This section contains some background
information on C*-subalgebras for convenience and coherence and also includes some modifications to
include non-commutative C*-subalgebras.

\medskip

Renault provided a notion of a Cartan subalgebra $A \subset B$ for the context of C*-algebras, called
Cartan pair $(A,B)$, with $A$ a maximal \emph{abelian} subalgebra (masa) \cite{Renault},
Previously, Kumjian  defined (abelian) C*-diagonals, a similar but less general notion \cite{C*diag}.
For more
information, see \ref{above} below. Kumjian showed that an invariant of C*-diagonals is a Fell line
bundle over a \emph{principal}
groupoid, and following those techniques, Renault showed that a Cartan pair invariant is given by a
Fell
bundle over a (locally compact Hausdorff) \emph{essentially principal} groupoid, thus widening the scope
of the examples. (An essentially principal groupoid $G_{ep}$ is defined as an \'etale groupoid in which
the interior of the
isotropy group bundle is equal to the object space $G_{ep}^0$ of $G_{ep}$. Equivalently,
the set of
points of $G_{ep}^0$ with trivial isotropy is dense. Every essentially principal
groupoid is
isomorphic to the groupoid of germs of a pseudogroup of locally defined homeomorphisms between the open
sets of some topological space $X$ \cite{Renault}.)

\medskip

Introducing a notion of non-commutative Cartan pair,
Exel addresses an even wider class of examples of regular C*-algebras, by constructing a Fell bundle
over a topological inverse semigroup for each generalised Cartan pair \cite{non com
cartan}. \footnote{Exel seems to have made this choice
perhaps because he was not interested in examples of Fell bundles over groupoids with
non-commutative fibres over $G^0$.}).

%In \cite{non com cartan},
%Exel already provided a non-commutative generalisation of Cartan subalgebras for von
%Neumann algebras, where the base space of the Fell bundle used in the construction of examples is a
%topological inverse semigroup (as opposed to a topological groupoid $G$

%To illuminate their structure and characterise these C*-algebra embeddings, Renault constructed an
%essentially principal groupoid and a 2-cocycle from each Cartan pair $(A,B)$. 

%Then, we use the
%information this gives us about a C*-bundle dynamical system, we %construct an essentially principal
%groupoid and a 2-cocycle $\omega$.
%\\

\subsection{Pure states and faithful conditional expectations} 
 
\begin{definition}\cite{fce} \label{P}
Let $B$ be a C*-algebra and let $A \subseteq B$ be a C*-subalgebra. Then we call $P:B \to A$ a
conditional expectation of $B$ onto $A$ if it satisfies the following three properties:
\begin{enumerate}
 \item $P(a)=a$, $\forall a \in A$;
 \item $P(a_1 b a_2) = a_1 P(b) a_2$  $\forall b \in B$, $\forall a_1,a_2 \in A$;
 \item $b \geq 0 ~\implies~ P(b) \geq 0$ $\forall b \in B$.
\end{enumerate}
We say that $P$ is a faithful conditional expectation if, in addition, $P(b^*b) \neq 0$ for all
non-zero $b \in B$. 
\end{definition}

Note that in definition \ref{P} there is no declaration that $A$ should be commutative. Let $p$, $p_i$
be central projections of $B$ for some Cartan pair $(A,B)$ \emph{such that $A$ is not
necessarily commutative}. If $A$ is non-commutative, then the projections $p_i$ may vary in
rank. The following are examples of faithful conditional expectations of $B$ onto $A$.

\begin{examples} \cite{fce,C*diag} \label{Examples P}
 \begin{enumerate}
  \item $P(b) = pbp + (1-p)b(1-p)$, $b \in B$.
  \item $P(b) = \sum_i p_i b p_i$ where $(p_i,...,p_n)$ is an $n$-tuple of pairwise orthogonal
projections of possibly varying rank.
  \item  Note that in this example, $A$ is abelian. 
  
  Let $B = M_n(\bbc)$, the algebra of complex $n$ by
$n$ matrices. Choose a set of matrix units,
$\{ e_{ij} : 1\leq ~ i,j ~ \leq n \}$ (one has $e_{ik} = e_{ij}e_{jk}$ and $e_{ij}^*=e_{ji}$), and let
$A$ denote the diagonal subalgebra (viz, $A$ is spanned by the $e_{ii}$s). Then $a = \Sigma \lambda_{ij}
e_{ij}$ normalises $A$ if and only if for each $i$, $\lambda_{ij} \neq 0$ for at most one $j$, and for
each $j$,  $\lambda_{ij} \neq 0$ for at most one $i$ (i.e., at most one entry is non-zero in each row
and column). If $i \neq j$, $e_{ij} \in N_f(A)$.  
 Let $P : B \to A$ be given by:
 \begin{equation}
  P(b) = \Sigma e_{ii} b e_{ii}.
 \end{equation}
This defines a faithful conditional expectation for which:
\begin{equation}
 \kernel P = \ls N_f(A).
\end{equation}
where $\ls N_f(A)$ denotes the closed linear span of the free normalisers of $A$ in $B$.
\end{enumerate}
\end{examples}

\begin{example} \label{Example P}
Let $\alpha$ be an action specifying a Fell line bundle $(E,\pi,G)$ as described in \ref{FB
example 1}, where $G$ is an essentially principal groupoid over a locally compact space $G_0 = X$.
Recall that $(A=C^*(E^0), B=C^*(E))$ is a Cartan pair. (Note that $\alpha_{g^*g}(b) = \alpha_x(b)$ is
not defined if
$b$ is not in $A$.) Then,
\begin{equation} \label{P nc}
 P(b) = \int_X \alpha_x(b) dx
\end{equation}
is the unique faithful conditional expectation from $B$ onto $A$.
 \end{example}

Let $A$ be a (resp. unital) C*-algebra. Recall that the space of pure
states $X(A)$ of $A$ is (resp. locally) compact in the weak *-topology. (We identify $X(A)$ with its
image $X$). 
\begin{definition}\cite{C*diag}
 If $(A,B)$ is a Cartan pair with abelian C*-subalgebra $A$, the ambient algebra $B$ is said to have the
\emph{extension property} if the pure states of $B$ restrict to the set
$X(A)$, that is, for $x \in X(A)$, if $x \circ P(B) = X(A)$. Equivalently,  $B = A + [B,A]$.
\end{definition}

The extension property ensures the
existence of a conditional expectation $P$ such that $x \circ P(B) = X(A)$. Let $A$ be abelian. If $B$
has the extension property, it follows that the abelian C*-subalgebra $A$ is maximal.

\medskip

Following \cite{C*diag}, for $n \in N(A)$, put

\begin{eqnarray}
 s(n) = \{ x \in X : x(n^*n) > 0 \}, \\
 I(n) = \{ a \in A : x(a) \neq 0 \implies x \in s(n) \}.
\end{eqnarray}
Note that  $s(n)$ is open in $X$ and $I(n)$ is an ideal in $A$.
Each $n \in N(A)$ defines a partial homeomorphism of $X$, 

\begin{equation} \label{partial homeo}
 f_{0,n} : s(n) \to s(n^*).
\end{equation}

\medskip

Finally, recall that the span of a possibly infinite
dimensional vector space $V$, is the set of all finite linear combinations of vectors $v \in V$.
Similarly, since $B = \ls N(A)$ where span denotes closed linear span, any element of $B$ can be
approximated by a set of finite linear combinations of normalisers (this was referred to in
\cite{C*diag} and will be useful to us below).

\subsection{Cartan pairs and C*-diagonals}

\begin{definition} \cite{Renault} \label{Renault def}
Let $A$ and $B$ be C*-algebras and let $A \subset B$. $A$ is said to be a \emph{Cartan
subalgebra} if:
\begin{enumerate}
 \item $A$ contains an approximate unit of $B$;
 \item $A$ is maximal abelian;
 \item $A$ is regular in $B$;
  \item There is a faithful conditional expectation of $B$ onto $A$.
\end{enumerate}
\end{definition}
If in addition, $\kernel P = \ls N_f(A)$, then $A$ is said to be \emph{diagonal} in $B$ (\cite{C*diag}).

\medskip

In \cite{Renault}, Renault shows uniqueness of the faithful conditional expectation $P:B \to A$ for
Cartan pairs $(A,B)$ and in \cite{ncg and sm}, Exel warns that for non-commutative algebras, there may
be many inequivalent probability measures. For this reason, he imposes a condition to ensure that the
faithful conditional expectation associated to a given non-commutative Cartan pair must be unique.
Moreover, in some of the examples we give below, the set (we define) of projections is a feature of
the algebra itself. Below we impose uniqueness of $P$ as an axiom. 

\medskip

Our non-commutative generalisation is :-

\begin{definition} \label{A in B}
 Let $A$ and $B$ be C*-algebras $A \subset B$, where $A$ is not necessarily commutative and let $A$
contain the unit of $B$ or an approximate unit for $B$. $(A,B)$ is said to be a \emph{(not necessarily
commutative) Cartan pair} if $A$ is regular in $B$ and if there is a unique faithful conditional
expectation $P:B \to A$.

If in addition, $\kernel P = \ls N_f(A)$, then $A$ is said to be \emph{diagonal} in $B$.
\end{definition}

\begin{lemma}
 Let $A$ be a non-commutative C*-diagonal in $B$. Then $B$ satisfies the criterion for the extension
property relative to $A$, namely,  that $B = A + \overline{[B,A]}$. (The overline denotes closed linear
span.)
\end{lemma}
\begin{proof}
 This follows from: 
 \begin{itemize}
  \item $A$ is regular in $B$;
  \item $\kernel P + A =B$;
  \item $B$ is an $A$-bimodule.
 \end{itemize}
\end{proof}

\begin{examples} \label{Examples diag}
Here are some simple examples of non-commutative Cartan pairs to illustrate the previous definition.
%These examples are
%interesting in their own right, for several reasons that will unfold below and in the next section, and
%because they were not treated previously. 
\begin{enumerate} 
 \item Consider example \ref{Example discrete FB}, (imprimitivity Fell bundle), where $B=C^*(E)$ and
$A=C^*(E^0)$, building on example 3.5(2). Since $E$
is saturated, $A$ is regular in $B$. Identify the pure state space $X$ of $A$ with the (discrete) base
space $X$ of the C*-bundle $E^0$. There is a unique faithful conditional expectation of $B$
onto $A$, given by, 
\begin{equation}
 P(b) = \sum_i^m p_i b p_i
\end{equation}
where $(p_i,...,p_m)$ is an $m$-tuple of pairwise orthogonal projections, as in example \ref{Examples
P}(2). The
rank of each projection $(p_i,...,p_m)$ is given by the dimension of each fibre, 
\begin{equation}
 E^0_{ \{ x \} } =  p_{ \{ x \} } B p_{ \{ x \} }.
\end{equation}
In other words, we define the set $(p_i,...,p_m)$ to be the maximal projections in the
C*-algebras $E^0_i$. Clearly, $P: B \to A$ is unique.
   \item Consider a locally trivial saturated Fell bundle $(E,\pi,G)$ (as in \ref{FB example 1}(b)),
where $B=C^*(E)$ and $A=C^*(E^0)$ and let $X$ be a discrete space. Let $\{ p_x \}_{x \in X}$ be an
$n$-tuple of pairwise orthogonal projections in $A$ such that 
\begin{equation}
 \{ E^0_x \}_{x \in X} = \{~ p_x E^0 p_x ~~ \vert ~~ \rank p_x = \dim E^0_x, ~~\forall x \in X ~\}
\end{equation}
One identifies the fibres $E^0_x = p_x E^0 p_x = \ls_{a \in A} \{ p_x a p_x \}$ of
$E^0$ with the unitary equivalence classes of irreducible representations, or pure
states, of $A$. Since $E^0$ is a trivial C*-bundle over $X$, the
projections $p_x$ are of constant rank, $\rank p_x = \dim E^0_x, ~~\forall x \in X$. From this we
interpret the probability measure to be unique on $X$. Now one makes the generalisation for
locally trivial C*-bundles $E^0$ over locally compact spaces $X$, because $X$ has unique probability
measure: the direct generalisation of equation \eqref{P nc} is, 
\begin{equation} 
 P(b) = \int_X \alpha_x(b) dx
\end{equation}
so that $P$ is the unique faithful conditional expectation from $B$ onto $A$. (Note that $\alpha$
determines an action of the groupoid $G$ on $B$, in which $\alpha_{g^*g}(b) := 0$ when $b$ is not a
member of $A$.)
\end{enumerate}
\end{examples}

\section{C*-bundle dynamical systems}

In this section, C*-bundle dynamical systems are introduced and their relationship to other
C*-dynamical
systems is studied, especially Arveson's $A$-dynamical systems.  The key differences that distinguish
C*-bundle dynamical systems from more general C*-dynamical systems are, (i) the covariance condition and
(ii) additional geometrical data encoded in the action of a 1-parameter group of diffeomorphisms
lifted to the C*-bundle. We include a discussion of the advantages to physics.

\medskip

Orientability of Fell bundles (and spectral
triples) is commented on and also connections between Exel's slices and Arveson's paths, are
discussed below. 

\medskip

A more direct
application of C*-bundle dynamical systems to C*-subalgebra theory is developed later in the
final section.

\subsection{Three families of C*-dynamical system}

To begin with, recall some fundamental definitions from (non-relativistic) non-commutative
C*-dynamical systems following the work of Raeburn \cite{Raeburn}, Williams,
Arveson and many others and then we modify them for the case of an action on a C*-bundle:

\begin{definition} \label{CDS}
 A C*-dynamical system is a triple $(B,\G,\alpha)$ where $\alpha$ is a strongly continuous action of a
locally compact group $\G$ on a C*-algebra $B$.
\end{definition}

For the action, we make a choice of a representation of $\G$, \\ 
$\alpha_g(b) = U_g b U_g^*$ for each $b \in \B$, where:
 
 \begin{equation}
  U_gU_h = \omega(g,h)U_{gh}
 \end{equation}
   for a 2-cocycle $\omega$, 
 \begin{equation}
  \omega : \G \times \G \to \mathbb{T}
 \end{equation}
  giving an action $\alpha$ of $\G$ on $\B$:

\begin{equation}
 (g_1,b_1)(g_1,b_2) = (g_1g_2, \alpha_{g_1}(b_1) b_2) \in \G \ltimes \B
\end{equation}

Note that $\omega$ measures the departure of the assignment $g \mapsto U_g$ from being a group
homomorphism i.e. representation of $\G$. The elements of $H^2(\G,\mathbb{T})$ consist of
equivalence classes of 2-cocycles $[\omega]$  (equivalent if they only differ by a coboundary). For
details see \cite{Raeburn}.

\medskip

%Intuitively, $\omega$ measures the departure from a
%cartesian product of $\G$ and $X$ residing in the semidirect product, so it measures how much the
%fibres fan out due to the curving of $X$. 

In the previous definition \ref{CDS}, the dynamical system is reversible. (A reversible dynamical system
is characterised by an action of a group and an irreversible system by an action of an
inverse semigroup, exactly because a group has inverses and an inverse semigroup only has
quasi-inverses.) 

\medskip
 
Next we study Arveson's $A$-dynamical systems, which take into account the structure of C*-algebras in
terms of their C*-subalgebras. Since we are using the symbol $B$ for the algebra instead of $A$, we can
think of A as standing for Arveson, rather than change to $B$-dynamical system.
 
 \begin{definition}\cite{Arveson} \label{A-DS}
  Arveson's \emph{A-dynamical system} is a triple $(\iota, B, \alpha)$ consisting of a semigroup $\alpha
= \{
\alpha_t : t \geq 0 \}$ of *-endomorphisms acting on a C*-algebra $B$ and an injective *-homomorphism
$\iota: A \to B$, such that $B$ is generated by $\cup_{t \geq 0} \alpha_t(\iota(A))$, where $A$ is a
C*-subalgebra of $B$.
 \end{definition}
In Arveson's definition there is no continuity requirement of the semigroup with respect to $t$. Of
course in the case of a \emph{reversible} $A$-dynamical system the $\alpha_t$ will form a group.

\medskip

$B$ is the norm-closed linear span of finite products,

\begin{equation}
 B = \ls \{ \alpha_{t_1} (a_1) \alpha_{t_2} (a_2)...\alpha_{t_k} (a_k) \}
\end{equation}
where $t_1,...t_k \geq 0$, $a_1,...,a_k \in A$, $k=1,2,...$.
For different times $t_1 \neq t_2$, the C*-algebras $\alpha_{t_1} (A)$ and  $\alpha_{t_2} (A)$
do not commute with eachother.

\medskip

Let $M$ be a smooth (or possibly discrete) manifold. Our modification to C*-bundles:-
\begin{definition} \label{CBDS}
 A \emph{reversible C*-bundle dynamical system} $(E^0,\G_{\sigma})$ is given by a C*-bundle
$(E^0,\pi,M)$ over a differentiable manifold $M$ and a strongly continuous 1-parameter covariance group
$\G_{\sigma}$ of C*-bundle spatial automorphisms. ($\G_{\sigma}$ was defined earlier in
\ref{covariance group}.)
\end{definition} 

\begin{theorem}
 Reversible C*-bundle dynamical systems provide a class of examples of reversible A-dynamical
systems. 
\end{theorem}
\begin{proof}
 Let $(E^0,\G_{\sigma})$ be a C*-bundle dynamical system with enveloping algebra $A=C^*(E^0)$ where
$E^0$ is a C*-bundle over the pure state space
$X$ of $A$, and let $G= X \times X$ (we assume that $X$ is a locally compact Hausdorff space $M$
admitting a differentiable
structure). The ambient algebra $B$ for the A-dynamical system
is given by $B=\G \ltimes_{\alpha} A$ where the action $\alpha$ of the group $\G = \Bis(G)$ on $A$ is
determined by the unitary representation $\rho_U$ on $E^0$ associated to $\G_{\sigma}$ as follows. See
also the earlier
explanation of these terms in section \ref{spatial section}. Since $\G_{\sigma}$ acts minimally
(is densely transitive) on the space of fibres $\{ E^0_x \}_{x \in M}$ of $E^0$, we can associate to it
a (strongly constinuous) unitary representation $\rho_U : \Bis(G) \to \SpatialAut(E^0)$, where
$G=M \times M$ and $\Bis(G)$ is identified with $\Diff(M)$.

The slices $\alpha_t(a)$ are then obtained from the representation $\rho_U$, by restricting $\rho_U$
to
$g \in \Bis_{\mathrm{herm}}(G)$ (definition \ref{Bisherm}) and then forming the closed linear span,
$\M = \ls_{a \in A} a U_g$ for $U_g = \rho_U \vert_{g \in \Bisherm}$. The
action $\alpha$ extends to an action $\alpha = \{ \alpha_t : t \geq 0 \}$ on $B$ such that $\cup_{t \geq
0} \alpha_t(\iota(A))$ generates $B$ with $\iota: A \to B$ an injective *-homomorphism. 
\end{proof}

Note that in a C*-bundle
dynamical system, there is a clear distinction between the configuration algebra $A$ and the observable
algebra $B$, whereas in an A-dynamical system, the $\alpha_t$ are interpreted as
alternative configuration spaces. 

\medskip

In an \emph{irreversible} C*-bundle dynamical system, the covariance group $G_{\sigma}$ should be
replaced by a inverse semigroup of *-endomorphisms of $E^0$ which preserve the bundle structure of
$E^0$. The ambient
Fell bundle $E$ will then be regular but not saturated and the fibres of $E^0$ will not usually
be isomorphic. We begin a discussion on the generalisation to the irreversible case in an additional
section, with which we close this paper.

\subsection{A remark on Arveson paths and Exel slices} \label{slice}

%Here we remark on the relationship between Exel's slices, Arveson's paths $\alpha_t(a)$, and the
%$A$-bimodules $\Omega_D^1(A)$ of 1-forms for finite dimensional C*-algebras $A$ in the non-commutative
%differential calculus.

Let $A$ be a regular C*-subalgebra in a C*-algebra $B$,
\begin{definition} \cite{non com cartan}
  A \emph{slice} is any closed linear space $\M
\subseteq N(A) \subset B$ such that both $A \M$ and $\M A$ are contained in $\M$.
\end{definition}

Observe that the $\alpha_t(a)$ in definition \ref{A-DS} are closed linear subspaces $\M$ of $N(A)$ such
that $A\M = \M$
and $\M A = \M$. Since $\M^* \M = A$ and $\M \M^*=A$, each $\M$ is a Hilbert $A$-bimodule. This means
that for each $t$, $\alpha_t(a)$ satisfies the definition of a \emph{slice} $\M$.

\medskip

The unital saturated Fell bundle over a discrete groupoid of example \ref{Example discrete FB}
describes a C*-category of Morita equivalence
(or imprimitivity) bimodules. The
space of fibres of the Fell bundle $(E,\pi,G)$ (as opposed to the diagonal bundle $E^0$) therefore forms
a structure that obeys
a set of rules equivalent to the defining axioms for groupoids but up to isomorphism. (Including
$E_g
\otimes E_{g^*} \cong E_{gg^*}$, $E_{g_1} \otimes E_{g_2} \cong E_{g_1 g_2} ~ \forall (g_1,g_2) \in
G^2$, $E_g \otimes E_{g^*g} \cong E_g$, $E_{gg^*} \otimes E_g \cong E_g$.) Finally observe
that the Arveson ``path'' (see \cite{Arveson paths}) defined by a closed linear space
$\alpha_t(a)$, is exactly a self-adjoint bisection of this ``weak'' groupoid. 

\medskip

Let $\sigma$ denote the generator of $\G_{\sigma}$. Note that $\sigma = \rho_U
(g(\lambda_0))$. In the case that $M$ is discrete and $\sigma$ is a bounded linear operator, the slices
$\M$ or Arveson paths are given by $\alpha_t(a) = \ls_{a \in A} [\sigma, a]$. (We leave the association
to the non-commutative differential calculus for finite spectral triples $(A,\H,D = \sigma)$ for another
chapter.)

%\in \{ U_t \}_{t \in\mathbb{R}} \}$
%(since Stone's theorem provides us
%with the limit $\mathrm{lim}_{t \to 0} t^{-1} (\alpha_t (a) - a) \forall a \in A$),
%Let $g=g^* \in \Bis(G)$ and

 %There is a $\{
%\alpha_{\lambda} : \lambda \geq 0 \}$ be a subgroup   of the group $\InnAut(E^0)$ of inner
%automorphisms
%of the C*-bundle $E^0$. ....

\section{Embeddings} \label{Embeddings section}

Following Kumjian's
observation that the normalisers $N(A)$ of $A$ in $B$ characterise the way $A$ \emph{embeds}
in $B$ (\cite{C*diag}), we define an ``embedding
invariant'' $\Phi_{\hookrightarrow}$ to be equivalent to a groupoid 2-cocycle $\omega$, and involving
$N(A)$. This characterises the crossed product algebras of the form $B = G
\ltimes A$ for $G=X \times X$, the pair groupoid over the state space $X$ of a possibly non-commutative
C*-algebra $A$ such that $A = C^*(E^0)$ for some C*-bundle $E^0$ with fibre
$M_n(\bbc)$. C*-bundles over discrete spaces in which the dimension of the fibres is
not constant, are also included.

 \subsection{Embedding invariant}

 Since $\Phi_{\hookrightarrow}$ is to be constructed from a groupoid 2-cocycle $\omega$, we first
explain the relevance of the latter.

\medskip

%We have stated the conditions for a non-commutative unital C*-algebra $A$ to embed as a Cartan
%subalgebra in a C*-algebra $B$. Examples are given to
%follow the definition.
Renault obtained a characterisation of Cartan pairs $(A,B)$ (definition \ref{Renault def}) in the form
of a groupoid 2-cocycle, 
\begin{equation}
 \omega: G_{ep} \times G_{ep} \to \mathbb{T}
\end{equation}
where $G_{ep}$ is an essentially principal
groupoid and $\mathbb{T}$ denotes the 1-dimensional unitary matrices. It follows that for each example
of a Cartan pair $(A,B)$ (definition \ref{Renault def}) one can define the multiplication in $B$ from
this 2-cocycle data in terms of a
semidirect product of $G_{ep}$ with $A$, as demonstrated in example \ref{FB example 1} (saturated Fell
bundle). So a
Fell line bundle over an essentially principal groupoid $G_{ep}$, defined with an action $\alpha$ is
equivalent to an essentially principal groupoid $G_{ep}$ together with a 2-cocycle $\omega$.

\medskip

More in detail, a twisted action $(\alpha, \omega)$
is given by an assignment of unitaries $g \mapsto u_g \in N(A)$ to each $g \in G_{ep}$, which does not
define a
representation of $G_{ep}$ because 
\begin{equation} \label{om}
 u_g u_h = \omega(g,h) u_{gh}
\end{equation}
where $u_{gh}$ is the composition of two unitaries $u_h \circ u_g$
implementing the action $\alpha$ of $G_{ep}$ on $A$, or equivalently, the representation of $G_{ep}$ on
$E^0$. Therefore, an embedding invariant
$\Phi_{\hookrightarrow}$
equivalent to a 2-cocycle mapping $\omega: G_{ep} \times G_{ep} \to \mathbb{T}$, can be presented as a
continuous assignment of
unitaries $g \mapsto u_g \in N(A)$ satisfying \eqref{om}.  

\medskip

In the special case where the groupoid $G$ is principal, then $A$ is diagonal in $B$. Note that since
given any
C*-diagonal $A$ in $B$, a Fell line
bundle $E$ over $G$ (as shown in \ref{FB example 1}) can be found such that
$A = C^*(E^0)$ and $B=C^*(E))$, the action $\alpha$ of $G$ on $A$, which controls the
multiplication in $E$, corresponds to a representation
$\rho_G$ of $G$ on $E^0$. Secondly, from $\rho_G$ one obtains a unitary representation $\rho_U :
\Bis(\rho_G) \to \SpatialAut(E^0)$ implemented by a mapping $g \mapsto U_g \in N(A)$ for each $g \in
\Bis(G)$, (where the mapping $g \mapsto U_g$ is induced directly from the previous mapping $g \mapsto
u_g$ for each $g \in G$). 

\medskip

Let $G$ be a groupoid and let $\G$ be the group $\G = \Bis(G)$. Note that it may be the case that a
group 2-cocyle for $\G$ may be a coboundary (that is, defines a representation of $G$, with $U_g U_h =
U_{gh}$) while a groupoid 2-cocyle is not, that is, it satisfies $u_g u_h = \omega(g,h) u_{gh}$ for
non-trivial $\omega$. 

\medskip

Now we modify proceedings to the non-commutative case. Let $A$ be a not necessarily commutative diagonal
C*-subalgebra in a C*-algebra
$B$. Making use of the normalising set $N(A) \subset B$, below we define the embedding invariant
$\Phi_{\hookrightarrow}$ to characterise such pairs $(A,B)$.  In the case
that $A=C^*(E^0)$ for $E^0$ a C*-bundle with isomorphic fibres, we define $\Phi_{\hookrightarrow}$ to be
equivalent to a principal groupoid 2-cocycle,

\begin{equation} \label{principal}
 \omega : G \times G \to \bigoplus_{i=1}^n \mathbb{T}.
\end{equation}
where $n$ is the dimension of the fibre of $E^0$. 
%Otherwise the codomain of the map $\omega$ will be the inverse category of partial isometries
%introduced
%and defined later below, \ref{ISL}.

\begin{definition} \label{embedding}

Let $A$ be a (possibly non-commutative) C*-diagonal in a C*-algebra $B$ where both $A$ and $B$ operate
on a separable Hilbert space $\H$. Since $\H$ is separable, it has a countable orthonormal basis  $\{p_i
\}_{i \in X}$ whose linear span is dense in $\H$.

\medskip

Let $u_{(i,j \neq i)} \in N_f(A)$ and $u_{(i,i)} \in A$ satisfy:
\begin{equation}
  p_k u_{(i,j)} p_l =  \begin{cases}
   u_{ij}, & \text{if $(k,l) = (i,j)$}.\\
    0, & \text{otherwise}.
  \end{cases}
\end{equation}
where each $u_{ij}$ is a partial isometry,  $u_{ij} \H_j \to \H_i$ where $\H_i=p_i \H$, $\H_j = p_j \H$.
(In the case that the rank of the projections $p_i$ is constant, each $u_{ij} : \H_j \to \H_i$ is an
isometry.)

\medskip

Given a choice of action $\alpha$, the
\emph{diagonal embedding} $\Phi_{\hookrightarrow}$ of $A$ in $B$ is fixed by an
approximation of a maximal rank element of $B$:-

\begin{equation} \label{phi G equation}
 \Phi_{\hookrightarrow}~=~ \{~~~ \sum_{(i,j)} u_{(i,j)}  ~ ~~\vert ~~~(i,j) \in Y \times Y \subset X
\times X ~~\} ~~\in B 
\end{equation}
for all subsets $Y \times Y \subset X \times X$.

%\medskip

%Note that the principal groupoid $G$ in \eqref{principal} being a pair groupoid $G=X \times X$ is
%equivalent to $X$ being simply connected. \flushright{$\diamond$}
\end{definition}

%Let $(A,B)$ be a not necessarily commutative Cartan pair as above.

%Let $\alpha$ be an action of the group $\G=\Homeo(X)$ on $A$, so that $B =
%\G \ltimes_{\alpha} A$. 

%The choice of action $\alpha$ is implemented by a mapping $g
%\mapsto U_g$ obeying the rule $U_g U_h = \omega(g,h) U_{gh}$ which is specified by a choice of group
%cocycle $\omega$. $\omega$ can be obtained either from a groupoid 2-cocycle or from a unitary
%representation $\rho_U$ specifying the action $\alpha$.

The above definition may be  generalised to Cartan pairs by replacing the pair
groupoid $\{ p_i \} \times \{ p_i \}$,  with the effective pair groupoid, which
can be constructed from $N(A)$ using \eqref{partial homeo}: $G_{ep}$ is the groupoid of germs of the
partial homeomorphisms defined by \eqref{partial homeo}. In this case, even if $\H$ is finite
dimensional, then $\Phi_{\hookrightarrow}$ can no longer be presented as matrix. We include this
generalisation as an Appendix.

%The proposition above is also
%generalised to Cartan pairs, where the $u_{(i,j)}$ are indexed by arrows in \ref{Gep} instead of arrows
%in $X \times X$.

\begin{examples} \label{Examples Phi}
 \begin{enumerate}
  \item Let $A$ be a non-commutative finite dimensional C*-algebra, which is diagonal in a
C*-algebra $B$. This automatically provides a C*-bundle over $X$, as described in
3.5(2) and where $X$ is interpreted as the space of pure states of $A$. 
  \begin{equation} \label{previous}
   \Phi_{\hookrightarrow} = \sum_{(x,y) \in X \times X} u_{(x,y)} ~~~(x,y) \in X \times X ~~~ \in B
  \end{equation}
with
\begin{equation}
 u_{xy} = p_x u_{(x,y)} p_y ~~~ \forall (x,y) \in X \times X.
\end{equation}
where each $u_{xy}$ is a partial isometry $u_{xy} \in M_{rs}(\bbc)$ where $r= \rank
p_x$ and  $s= \rank p_y$. In this case, $\Phi_{\hookrightarrow}$ is  a
state-transition matrix for algebraic quantum gravity.
\item Let $E^0$ be a locally trivial C*-bundle over $X$ with non-commutative fibres and let $E$ be a
Fell bundle over $G=X \times X$ as in example \ref{FB example 1}(b) (Saturated Fell bundle). 
$\Phi_{\hookrightarrow}$ as in equation \eqref{phi G equation} where each $u_{ij}$ is a unitary
matrix whose rank and dimension is equal to the dimension of the fibres
of $E^0$.
\item Let $E^0$ be a 1-dimensional C*-bundle over a compact space $X$ where $B=C^*(E)$ and
$A=C^*(E^0)$. Identify $X$ with the image of $X(A)$. Each $u_{ij} \in \mathbb{T}$ where $\mathbb{T}$
denotes the 1-dimensional unitary matrices. 
 \end{enumerate}
\end{examples}

\begin{comment} \label{embed}
 This notion of embedding for the context of C*-algebras is motivated by the fact that in topology, an
embedding is
a homeomorphism onto its image. In analogy, the way a C*-algebra $A$ embeds in a C*-algebra $B$ is
determined by inner automorphisms (non-commutative homeomorphisms) implemented by the unitary
normalisers of $A$ in $B$.
\end{comment}

\begin{remark}
 In summary, the embedding invariant $\Phi_{\hookrightarrow}$ provides all the information about the
Fell
bundle and the way in which $A$ embeds in $B$. It might be helpful to keep in mind its form as a
transition matrix. The following data can be read-off directly from  $\Phi_{\hookrightarrow}$: $\alpha$,
$\omega$, $(E,\pi,G)$, $(E^0,\pi,G_0)$, $A \subset B$, $N(A)$, $N_U(A)$, $P : B \to A$,
$\rho_U$ and $\rho(G)$.
\end{remark}

Given a possibly non-commutative diagonal C*-subalgebra $A$ in $B$, the embedding invariant
$\Phi_{\hookrightarrow}$ can be generated by a covariance subgroup $\G_{\sigma}$ of the associated
unitary representation $\rho_U$, as follows.

\begin{theorem}
Let $(E^0,\G_{\sigma})$ be a C*-bundle dynamical system and let $(E,\pi,G)$ be a saturated Fell bundle
over a principal groupoid over a (possibly discrete) differentiable manifold $M$. The embedding
invariant $\Phi_{\hookrightarrow}$ associated to $(A,B)$  can be
constructed from the abelian subgroups $G_{\sigma} \subset \rho_U$. 
\end{theorem}

\begin{proof}
(i) Consider first a discrete space $M$ with $n$ points $x \in M$. Then by inspection of definition
\ref{embedding},

%Consider the assignment $g
%\circ g \circ...\circ g \mapsto U_g U_g...U_g$,

\begin{equation} \label{partition}
\Phi_{\hookrightarrow} = \sum_{m=1}^n \prod_{i=1}^m U_{g_i}, ~~~ i=1,...,m, ~~m=1,...,n.
\end{equation}

because since $\Diff(M)$ is generated by a finite diffeomorphism $g$ or $g(\lambda_0))$, a 1-parameter
group of diffeomorphisms is a finite product group, that is, $ \{ g(\lambda) \}_{\lambda \in
\mathbb{Z}}$ and since $\sigma = \rho_U(g(\lambda_0))$, elements $(f_0,f,U_{f_0})$ or
$(g(\lambda),f,U_{g(\lambda)})$ of $\G_{\sigma}$ are obtained by self-multiplications of
$\rho_U(g(\lambda_0)$. 

%For the general case, $\Phi_{\hookrightarrow}$ is approximated by linear combinations of finite
%products of the generating element $U_{g_i}$.
More generally, since $A$ is regular in $B$, we use the fact that $B = \ls N(A)$ and
approximate $\Phi_{\hookrightarrow} \in B$ by sets of finite linear combinations of finite products of
normalisers $U_{g_i}$. %Note that for $X$ a continuous space, equation
%\eqref{partition} may not converge.
\end{proof}

\begin{remark} 
Note that $\G_{\sigma}$ and $\Phi_{\hookrightarrow}$ are embedding
invariants equivalent to a groupoid 2-cocycle $\omega$.
 \end{remark}
 
\begin{remark} \label{partition remarks}
 In quantum gravity, a partition function is a discretisation of the path integral in quantum
field theory. The algebra invariant $\Phi_{\hookrightarrow}$ is a topological invariant and plays the
r\^ole of a state-transition matrix. When diagonalised, it is
analogous to a partition function with constant weight for a spin foam on a discretised manifold, since
$\Phi_{\hookrightarrow}$ is a sum over geometrical states (or irreducible representations of the algebra
$A$) and the product is over parallel transports, where each $(f,f_0,U_{f_0}) \in \SpatialAut(E^0)$ can
be said to provide a system of parallel transports in the C*-bundle since it provides a set of isometric
*-isomorphisms between the fibres of $E^0$. 
\end{remark}

\begin{example}
For illustrative purposes, consider a discrete space $M$ with $n=4$ points. Let $G$
denote the pair groupoid over this discrete space and $\Bis(G)$ the group of global bisections. The
diagrams illustrate how we generate $\Phi_{\hookrightarrow}$ from actions of subgroups of
$\Bis(G)$ on a finite dimensional C*-bundle $E^0$ over $M$. Let $g$ be the following
transitive\footnote{In the discrete situation, there is no distinction between a minimal and
a transitive flow.} element of $\Bis(G)$, 

\vspace{2pc}
\begin{equation}
 g = \xymatrix{
 \cdot \ar@/^1pc/[r] & \cdot
 \ar@/^1pc/[r] & \cdot
 \ar@/^1pc/[r] & \cdot
 \ar@/^1pc/[lll] }
\end{equation}

\begin{equation}
 g \mapsto U_g,
\end{equation}

\begin{equation}     
U_g =
\left(   \begin{array}{cccc}
0         &          u_{12}     &   0        &   0\\
0         &          0          &   u_{23}   &   0\\
0         &          0          &   0        & u_{34}\\
u_{41}    &          0          &   0        &  0
\end{array}  \right)
\end{equation}
\vspace{1pc}
\begin{equation}
g \circ g = \xymatrix{
 \cdot \ar@/^1pc/[rr] & \cdot \ar@/^1pc/[rr] & \cdot \ar@/^1pc/[ll] & \cdot \ar@/^1pc/[ll] }
\end{equation}

\begin{equation}
 g \circ g \mapsto U_{g \circ g}
\end{equation}

%\vspace{1pc}

\begin{equation}
 U_{g \circ g}=\left(   \begin{array}{cccc}
0         &          0          &   u_{13}   &   0\\
0         &          0          &   0        &   u_{24}\\
u_{31}    &          0          &   0        &   0\\
0         &          u_{42}     &   0        &  0
\end{array}  \right)
\end{equation}

\vspace{1pc}
where for example, $u_{12}u_{23} = \omega(12,23)u_{13}$.
%\newpage

The self-adjoint elements of $\Bis(G)$ include:

\vspace{2pc}
\begin{equation}
 \xymatrix{
  \cdot \ar@/^2pc/[rrr] & 
  \cdot \ar@/^1pc/[r] & 
  \cdot \ar@/^1pc/[l]  & 
  \cdot \ar@/^2pc/[lll]   }
\end{equation}

\vspace{2pc}
\begin{equation}
 \xymatrix{
 \cdot \ar@/^1pc/[rr] & \cdot \ar@/^1pc/[rr] & \cdot \ar@/^1pc/[ll] & \cdot \ar@/^1pc/[ll] }
\end{equation}

\vspace{2pc}
\begin{equation}
 \xymatrix{
 \cdot \ar@/^1pc/[r] & \cdot
       \ar@/^1pc/[l] & & \cdot
       \ar@/^1pc/[r] & \cdot
       \ar@/^1pc/[l] }
\end{equation}
\end{example}
\vspace{1pc}

\begin{comment}[Orientability]
We have been implicitly assuming that $E^0$ is orientable since we have been interpreting
$\Phi_{\hookrightarrow}$ as a
nowhere vanishing global section. If $E^0$ is not orientable then 
$\Phi_{\hookrightarrow}$ is not an element of $B$. In that case, to treat $\Phi_{\hookrightarrow}$ with
a topological twisting of the bundle, we will have to insert $u_gu_h = \omega(1-u_{gh})$ for a pair of
elements
$g,h \in G$. If $(C^*(E^0), \H, \sigma, \gamma)$ is an even spectral triple $(A, \H, D, \gamma)$, then
the orientability condition for spectral triples $D \gamma = \gamma D$ \cite{gravity} (which is
automatically satisfied for $D=\sigma$ \cite{spectral C}) might be replaced by the condition that $E^0$
is orientable: that $\Phi_{\hookrightarrow}$ is an element of $B$. (See \cite{Christoph} for discussions
on the orientability condition for spectral triples.) In the context of the non-commutative standard
model and Fell bundles (\cite{spectral C}), a transition matrix (to counterpart
$\Phi_{\hookrightarrow}$) for a non-orientable Fell bundle might
indicate a topological defect in the vacuum manifold. 
\end{comment}

%Observe that the algebra $B$ of the Fell bundle $A*G$ corresponds to the crossed product Fell
%bundle algebra $\Bis_{\mathrm{herm}}(G) \ltimes A$.

\section{C*-bundle dynamical systems and Fell bundles} \label{above}

%\subsection{Non-commutative Cartan pairs and Fell bundles} 

Earlier in section \ref{non-ab}, we recalled some important results by
Kumjian,
 Renault and Exel on C*-subalgebras. More specifically, Kumjian established the correspondence between
pairs of C*-algebras $(A,B)$ such that $A$ is a C*-diagonal in $B$ and twisted Fell line bundles $E$
over principal groupoids $G$, \cite{C*diag} (although Kumjian defined Fell bundles over groupoids later
in \cite{fbg}). Following some of Kumjian's techniques, Renault generalised the situation to Fell line
bundles over (locally compact Hausdorff) essentially principal groupoids $G_{ep}$, with $C_r^*(E) \cong
C_r^*(G_{ep})$, where
$C_r^*(G_{ep})$ is the reduced C*-algebra completion of the groupoid convolution algebra. He
introduced a notion of Cartan pair $(A,B)$ for the context of C*-algebras and then established a
correspondence between them and essentially principal groupoids together with a
groupoid 2-cocycle $\omega$, which he obtained by the twisted groupoid convolution product
\cite{Renault}. The crossed product algebras $B$ involved in each of
these works, which are by construction regular, take the form $B= G \ltimes A$
(\cite{C*diag}) or $B=G_{ep} \ltimes A$ (\cite{Renault}).

\medskip

Cartan subalgebras that are not C*-diagonals:

Let $G_{ep}$ be an essentially groupoid over a locally compact space $X$. Let $(E,\pi,X)$ be a
Fell bundle over $G_{ep}$. Let $(A,B)$ be a Cartan pair with faithful
conditional expectation $P: B \to A$. Recall that $A$ is diagonal in $B$ if
and only if $\kernel P = \ls N_f(A)$. Observe that in this case of a Fell bundle over an essentially
principal groupoid as opposed to a principal groupoid, the kernel of $P : C^*(E) \to C^*(E^0)$ is a
larger set than the set $\ls N_f(C^*(E^0))$. It follows that $A:=C^*(E^0)$ is a Cartan
subalgebra of $B:=C^*(E)$, and not a C*-diagonal in $C^*(E)$. This is why the scope of examples treated
in
\cite{Renault} is larger, and in turn, why the classifications of C*-subalgebras defined by each of
Kumjian and Renault, are not equivalent. Two families of these ``extra'' examples
that have regular masas and are not C*-diagonals, include the graph C*-algebras and the Cuntz algebras
\cite{Renault}. For an extensive work on the automorphisms of Cuntz algebras, see \cite{Roberto} and see
\cite{Kumjian notes} for a survey on graph C*-algebras.

\medskip

Later, in Exel's work \cite{non com cartan}, a
correspondence between crossed product C*-algebras of the form $B=\Ss \ltimes A$ and Fell
bundles over inverse semigroups $\Ss$, with enveloping algebra $C^*(E)=B$ is established, where
$A=C^*(E^0)$ is a non-commutative C*-algebra, arising from the twisted convolution algebra of
$\Ss$. Exel underlines his motivation that N. Sieben’s notion of
Fell bundles over inverse semigroups should be thought of as twisted \'etale groupoids with
non-commutative unit space. 

\medskip

In this section we treat examples of locally trivial Fell bundles over groupoids that are not
necessarily one dimensional, and so their enveloping algebras $C^*(E)$ are not convolution algebras,
neither of groupoids nor of inverse semigroups. (Refer to examples 2.5.)  These algebras
$C^*(E)$ can be thought of as tensor product algebras $B = C^*_r(G) \otimes M_n(\mathbb{C})$. The
advantage of these examples to physics includes example 2.5(5), and in the case that $G$
is discrete, we refer to arguments that a realistic notion of space-time manifold $M$ is unlikely to be
either continuous or commutative (\cite{Isham}, \cite{wqg}, \cite{qg}). Another example of a
non-commutative diagonal pair is from \cite{CDS Roberto}: Let $\Sigma = (B,\G,\alpha, \omega)$ be a
unital discrete twisted C*-dynamical system (as studied in \cite{CDS Roberto}) such that $B = A
\rtimes_{\alpha} \G$ where $A$ is a simple C*-algebra and where $\G$ is a discrete subgroup of $\Bis(G)$
for $G = M \times M$, then $(A,B)$ is a not necessarily commutative diagonal pair.

%\subsection{C*-bundle dynamical systems as enveloping structures}

\subsection{Groupoid 2-cocycles as Cartan pair invariants}

The following generalises Renault's result that a one dimensional Fell bundle over an essentially
principal groupoid $G_{ep}$ specified by a groupoid 2-cocycle, is equivalent to his definition of a
Cartan pair $(A,B)$ of C*-algebras. In this generalised case, we consider locally trivial saturated Fell
bundles $(E,\pi,G_{ep})$ (from \cite{fbg}, saturated Fell bundles over locally compact
groupoids are equivalent to semidirect product example \ref{FB example 1}). Since these algebras
$C^*(E)$ can be thought of as tensor product algebras $B = C^*_r(G) \otimes M_n(\mathbb{C})$ it should
be possible to produce the following result using only Renault's construction but we provide these
alternative techniques because they have the potential to be generalised to non-saturated Fell
bundles over inverse semigroups and because the result leads into the next result which involves
C*-dynamical systems and the unitary representations as discussed in section \ref{spatial section}. 

\begin{theorem} \label{thm esspal}
(a) Each locally trivial Fell bundle $(E,\pi,\esspal)$ over an
essentially principal
groupoid $G_{ep}$ over a locally compact Hausdorff space $X$, gives rise to a not
necessarily abelian Cartan pair $(A,B)$ of separable C*-algebras (as in definition \ref{A in B}) such
that $B =
\esspal \ltimes_{\alpha} A$, where $\esspal$ is an essentially principal groupoid over $X(A)$. (b) All
such Cartan pairs, (with  $B = \esspal \ltimes_{\alpha} A$ for some essentially principal groupoid
$\esspal$ and some separable C*-algebra $A$), arise in this way.
\end{theorem}

\begin{proof}
(a)  First of all, the restriction of $E$ to the unit space $G^0_{ep}$ of $\esspal$ provides a
C*-bundle
$E^0$, whose enveloping C*-algebra $C^*(E^0)$, as detailed
above in section 2, will provide the C*-subalgebra $A$ in the pair $(A,B)$ that we are
constructing, $A:=C^*(E^0)$. 
 The particular embedding of $A$ in $B$ for the Cartan pair $(A,B)$, is specified by the semidirect
product structure as in examples \ref{FB example 1} above, providing the ambient algebra $B= G_{ep}
\ltimes_{\alpha} A$. There is a unique faithful conditional expectation $P : B \to A$, given by the
restriction map $P : C^*(E) \to C^*(E^0)$. 

(b) Let $(A,B)$ be a not necessarily commutative Cartan pair with unique faithful conditional
expectation $P : B \to A$. A C*-bundle $(E^0,\pi,X)$ is constructed from the
C*-subalgebra $A$ as follows. Since $A$ is separable, (and assuming that $A$ is faithfully represented
on the associated Hilbert space $\H$), we identify $A$ with $\bigoplus^X_m \pi_m$, the direct sum over
all irreducible representations $\pi_m$ of $A$. Define the space of fibres of $E^0$ to be given by $\{
E^0_x \}_{x \in X} = \{ \pi_m \}$, where $X$ is identified with the pure state space $X(A)$ which is
locally compact in the weak *-topology. Then the enveloping algebra $C^*(E^0)$ of $E^0$ is identified
with $A$. 

To construct a Fell bundle $(E, \pi, \esspal)$ from $(A,B)$ one defines an essentially
principal groupoid $G_{ep}$ as the groupoid of germs of the locally defined homeomorphisms on $X$ (see
\eqref{partial homeo}) and the Fell bundle $E$ over $\esspal$ is given by $E = \esspal \ltimes_{\alpha}
E^0$ such that $B= C^*(E)$, with $P : C^*(E) \to C^*(E^0)$ and where the representation of $A$
extends to a faithful representation on $\H = L^2(E)$. To see that all such Cartan pairs $(A,B)$ arise
in this
way, one constructs the groupoid 2-cocycle from the interaction between $B$ and $\esspal$ in order to
specify $E$. Recall the
earlier result \ref{N(A)=InnAut}, and note that
\begin{equation}
  \ls ( A ~N_U(A)) ~~=~~ \ls (A~ \Bis(\rho_{\esspal})) 
\end{equation}
for a certain representation $\rho_{\esspal}: \esspal \to \SpatialAut(E^0)$ of $\esspal$ on
$E^0$. (In the case that $G_{ep}$ is principal $G$ we have a representation $\rho_G : G \to
GL(E^0)$ inducing a representation $\Bis(\rho_G) = \rho_U : \Bis(G) \to \SpatialAut(E^0)$.)
Since $A$ is regular in $B$, 
\begin{equation}
 B ~= ~\ls (N(A))  ~~=~ ~ \ls ( A~ \rho_{\esspal})
\end{equation}

which illustrates that the data for the representation $\rho_{\esspal}$ and the action $\alpha$ come
from the multiplication rule in $B$, and in turn the action $\alpha$ gives the mapping $\omega : \esspal
\times \esspal \to \bigoplus_{i=1}^n \mathbb{T}$ (where $n$ is the dimension of the fibre of $E^0$) and
a Fell bundle is found such that $B=C^*(E)$. Finally, note that $E^0$ is
locally trivial and the map $P : B \to A$ (as in \eqref{P nc}) gives the
restriction of enveloping algebras $P : C^*(E) \to C^*(E^0)$. 
\end{proof}

\subsubsection{Bridge theorem}

Let $(E^0,\G_{\sigma})$ be a reversible C*-bundle dynamical system (definition \ref{CBDS}). In the
following theorem we make use of several constructions set out
previously in this paper, to show that $(E^0,\G_{\sigma})$ provides an enveloping structure
 capturing non-commutative Cartan pairs together with their respective Fell bundles, where the
embedding invariant $\Phi_{\hookrightarrow}$ creates a bridge between the two
descriptions. The Fell
bundles in question are not necessarily 1-dimensional and  we
use new techniques provided by unitary representations of groups and groupoids as discussed in
\ref{spatial section}.

\medskip

We assume that in all examples, the state space $X$ of $A$ is a locally compact Hausdorff space (or
topological manifold) admitting a smooth structure.

\begin{theorem}[Bridge] \label{bridge}
(a) Let $M$ be a locally compact simply connected manifold. Each reversible C*-bundle
dynamical system, $(E^0,\G_{\sigma})$ gives rise to a locally trivial Fell bundle over a connected
principal groupoid and its uniquely associated (not necessarily
abelian) Cartan pair. (b) All such Fell bundles over connected principal groupoids $G$
together with their uniquely associated Cartan pair, arise in this way.
\end{theorem}

\begin{proof}
 (a) In order to identify the Cartan pair $(A,B)$ and Fell bundle $(E, \pi, \esspal)$ that arises
from each $(E^0, \G_{ep, \sigma})$, one constructs an embedding invariant $\Phi_{\hookrightarrow}$ from
$(E^0, \G_{ep, \sigma})$ and then one reads-off all the required data from $\Phi_{\hookrightarrow}$.

 Treat first the case where $\H$ is a finite dimensional Hilbert space and $M$ is a discrete
manifold with $n$ points. Let $G$ be the pair groupoid over $M$ and let $g(\lambda_0)$
be the generating diffeomorphism of the finite product group $\{ g(\lambda) \}_{\mathbb{Z}} \subset
\Bis(\esspal)$ such that 
\begin{equation}
 \G_{\sigma} : \{ g(\lambda) \} \to \SpatialAut(E^0)
\end{equation}
Let $U_g$ denote the operators implementing the unitary representation given by
$\G_{\sigma}$. Since $\G_{\sigma}$ is minimal (over the space of fibres of
$E^0$) one writes:-
  
\begin{equation} 
\Phi_{\hookrightarrow} = \sum_{m=1}^n \prod_{i=1}^m U_{g_i}, ~~~ i=1,...,m, ~~m=1,...,n.
\end{equation}
 
 as in \eqref{partition}.

After having constructed $\Phi_{\hookrightarrow}$ for $(E^0, \G_{\sigma})$, we read-off from
$\Phi_{\hookrightarrow}$ all the required information about the C*-subalgebra $A =
C^*(E^0)$ and the way in which it embeds in the ambient C*-algebra $B=C^*(E)$ as determined by
$\alpha$, $\omega$ and $P$. More in detail, denoting $u_{(i,j)}$ by $u_{(p_i,p_j)}$, we approximate the
elements of
the algebras $A$ and $B$ from the following sets of finite sums,
\begin{gather}
 a ~=~ \{ ~~ \sum a_i  u_{(p_i,p_i)}  ~~ \vert ~~ (p_i,p_i)~\in ~(p_i,..,p_k) \times
(p_i,..,p_k) \} \in A \\
 b~ = ~\{ ~~ \sum a_i  u_{(p_i,p_j)}  ~~ \vert ~~ (p_i,p_j)~ \in ~(p_i,..,p_k) \times
(p_i,..,p_k) \} \in B
\end{gather} 
where $(p_h,.,p_n)$ is an $n$-tuple of pairwise orthogonal projections with $(p_i,..,p_k) \subset
(p_h,.,p_n)$ and where $a_i$ denote arbitrary elements of $p_i A$.

$B = \ls N(A)$ where $n \in N(A)$ is approximated by $n= \sum_{(i,j)} a_i u_{(i,j)}$ for
pairs $(p_i,p_j)$ and choices of $a_i \in p_i A$.

The set $(p_h,.,p_n)$ also provides the following.

$\cup_{p_i}^n \in A$ forms a unit or an approximate unit for $B$.

There is a unique faithful conditional expectation, $P:B \to A$,  $P(b) = \sum_i p_i b p_i \in A$ with
$\kernel P = N_f(A)$.

(b) Conversely, one shows that each Fell bundle $(E,\pi,G)$ together
with its
uniquely associated non-commutative Cartan pair $(A=C^*(E^0), B=C^*(E))$, arises as above. Present the
information specifying the action $\alpha$ and the 2-cocycle $\omega$ for $(A,B)$ and $(E,\pi,\esspal)$,
by associating an embedding invariant $\Phi_{\hookrightarrow}$ to $(A,B)$ and $(E,\pi,\esspal)$ with
restriction map $P : C^*(E) \to C^*(E^0)$. Then, from $\Phi_{\hookrightarrow}$, one
approximates a unitary representation $\rho_G : G \to GL(E^0)$ inducing $\rho_U =
\Bis(\rho_G)$, and then a
covariance group $\G_{\sigma} \subset \rho_U$ is obtained such that $\G_{\sigma}$ has minimal flow by
restricting $\rho_U$ to a 1-parameter subgroup of diffeomorphisms 
$\{ g(\lambda) \}_{\lambda \in \mathbb{R}} \in \Bis(G)$. In the case that the principal groupoid is
a pair groupoid, then $\G_{\sigma}$ should have minimal (densely transitive) flow. 
\end{proof}

\section{Pre-requisites for irreversible C*-bundle dynamical systems}

Here is some additional material that will be required in order to define the more general notion of
\emph{irreversible} C*-bundle dynamical system. This material is not
particularly new since Paterson (see the book \cite{Paterson}) already developed the theory of
partial isometry representations of inverse semigroups.

\medskip

Let $V$ be an operator on a Hilbert space $\H$. Recall that $V$ is a \emph{partial isometry}
if there exists a unique operator $V^*$ (which we call quasi-inverse) such that $VV^*V=V$ and also
$V^*VV^*=V^*$. Clearly, the invertible partial isometry operators are the isometries $U$, (satisfying
$U^*U=1$, $UU^*=1$). The set of all partial isometries $V$ on $\H$ form an inverse semigroup $\V(\H)$,
whereas the set of all unitary operators $U$ on $\H$ form a group $\U(\H)$.

\medskip

Strictly, a groupoid and an essentially principal groupoid are special cases of inverse
semigroups
$\Ss$, (a groupoid is an inverse category with a 0 element formally adjoined, in which the
quasi-inverses are true inverses.) %Even so, these examples were not treated specifically.

\begin{definition} \label{ISL}
 Let $(E^0,\pi,X)$ be a C*-bundle over a locally compact (possibly discrete) space $X$, with enveloping
algebra $A$. We define $ISL(E^0)$ the inverse category consisting of all involutive *-endomorphisms
$\alpha_x$ between fibres of $E^0$ such that for each $\alpha_x$, there is a unique $\alpha_x^*$
satisfying, $\alpha_x \circ \alpha_x^* \circ \alpha_x = \alpha_x$ and $\alpha_x^* \circ \alpha_x \circ
\alpha_x^* = \alpha_x^*$. Call $\alpha_x^*$ the quasi-inverse of $\alpha_x$.  
\end{definition}
Note that $ISL(E^0)$ has the structure of an inverse semigroup with 0 element formally adjoined.

\begin{definition}
Let $(E^0,\pi,X)$ be a C*-bundle over a locally compact space $X$, with enveloping algebra $A$,
faithfully represented on a separable Hilbert space $\H$. 
A \emph{C*-bundle spatial *-endomorphism} consists of continuous maps $f$ and $f_0$ with
commuting diagram:
 
 \begin{equation} 
  \xymatrix{
        E^0 \ar[r]^f \ar[d]_{\pi} & E^0 \ar[d]^{\pi} \\
        X \ar[r]_{f_0} & X  }, 
 \end{equation}
 
 \begin{itemize}
  \item such that each induced fibrewise map $f_x : E^0_x \to E^0_{f_0(x)}$ is continuous,
  %\item whenever $\parallel s_i \parallel \to 0$ and $\pi(s_i) \to x$ in $E$, then $f(s_i) \to 0_x$ in
 %$E'$,
  \item and such that $f$ extends to a *-endomorphism $\hat{f} : A \to A$ of the form  $\hat{f}(a) =
  VaV^*$ where $V$ is a partial isometry map on $\H$. 
 \end{itemize} \flushright{$\diamond$}
\end{definition}
Obviously, the invertible C*-bundle spatial *-endomorphisms are the C*-bundle spatial automorphisms
$\SpatialAut(E^0)$ (\ref{IA}). 

\medskip

The set of inner
*-endomorphisms of $E^0$ form an inverse category $\mathrm{InnEnd}_{\pi}(E^0)$, or an inverse semigroup
with 0 adjoined.

\begin{definition} \label{partial isom rep}
 A \emph{partial isometry representation} $\rho_V$ of an inverse semigroup $\Ss$ on a C*-bundle $E^0$ is
an
inverse semigroup homomorphism: $\rho_V : \Ss \to \mathrm{InnEnd}_{\pi}(E^0)$ such that $V$ is a partial
isometry on a separable Hilbert space $\H$. 
\end{definition}
Obviously, the invertible partial isometry representations $\rho_V$ are the unitary representations
$\rho_U$, \eqref{unitary rep}.

\begin{example}
 Let $(A,B)$ be a non-commutative Cartan pair such that $B$ is a crossed product algebra: $B=\Ss \ltimes
A$ where $\Ss$ is an inverse semigroup. We have already mentioned that $A$ is regular in $B$.  The
action of $\Ss$ on $A$ corresponds to a
partial isometry representation $\rho_V$ of $\Ss$ on $E^0$. This is the main family of examples
considered in \cite{non com cartan}.
\end{example}

\section{Appendix}

The following presentation of an embedding invariant $\Phi_{\hookrightarrow}$ for non-commutative Cartan
pairs of the form $B = \esspal \ltimes_{\alpha} A$, is equivalent to a groupoid 2-cocycle $\omega ~:~
\esspal \times \esspal \to \bigoplus_{i=1}^n \mathbb{T}$. This is useful as it allows one to readily
switch between  a Fell bundle specified by a 2-cocycle mapping
$\omega$ and its associated Cartan pair $(A,B)$. 

\begin{definition}[Embedding invariant, $\Phi_{\hookrightarrow}$] \label{esspal phi}
Let $(E,\pi,\esspal)$ be an orientable, (not necessarily 1-dimensional) locally trivial Fell bundle,
specified by 2-cocycle data $\omega : \esspal \times \esspal \to ISL(E^0)$.
And let $(A,B)$ be the
(possibly non-commutative) Cartan pair associated to $E$. Both $A$ and $B$
operate on a separable Hilbert space, $\H = L^2(E)$. 

Let $\{p_i \}$ be a countable
orthonormal basis for $\H$ which is dense in $\esspal^0$, given by maximal central projections
in $B$ as follows. Define simple matrix units $e_{xy} \in E_{(x,y)}$ for all $x \in X(A)$ such that the
rank of central projections $e_{xy}e_{xy}^*$ in $B$, satisfy: rank $e_{xy}e_{xy}^*$ = dim $(E^0_x)$
for each $x \in \esspal^0$. Unless $\esspal$ is principal (that is, if it has trivial isotropy) then
$e_{xy}e_{xy}^* \in \mathrm{InnAut}(E^0_x)$ is not identified with $e_{xz} e_{xz}^* \in
\mathrm{InnAut}(E_x^0)$, which explains the need for the additional index $i$.

%in the automorphism groups of fibres $\{ \mathrm{Aut} ( E^0_x ) \}_{x \in X}$ of $E^0$. Since
%$B$ has the unique extension property, $X(A) \subset X$. 

\medskip

We have, $\esspal \cong \{ e_{xy}, e_{xy} e_{xy}^* \}$ and $\{ p_i \} = \{ e_{xy} e_{xy}^* \}$.

\medskip

In the special case that the base space of $E$ is a principal groupoid $G$, then $\kernel P = \ls
N_f(A)$ and in this case $A=C^*(E^0)$ is a C*-diagonal subalgebra of $B=C^*(E)$. More generally, the
base space of $E$ is an essentially principal groupoid $\esspal$, and then, 

\begin{equation}
 \kernel P = \ls_{n \in N_f(A)} \{ n, nn^* \}
\end{equation}

Let $u_{(i,j \neq i)} \in N_f(A)$ and $u_{(i,i)} \in A$ satisfy:
\begin{equation}
  p_k u_{(i,j)} p_l =  \begin{cases}
   u_{ij}, & \text{if $(k,l) = (i,j)$}.\\
    0, & \text{otherwise}.
  \end{cases}
\end{equation}
where each $u_{ij}$ is an isometry, $u_{ij} : \H_j \to \H_i$, where $\H_i := p_i \H$.

Given the action $\alpha$, we have an assignment,

\begin{equation}
 g \mapsto u_g = p_i u_{(i,j)} p_j.
\end{equation}
 where  $u_g$ denotes the unitaries implementing the action $\alpha$, which is in turn specified by the
2-cocycle $\omega ~:~ \esspal \times \esspal \to \bigoplus_{i=1}^n$. (Note that the mapping $g \mapsto
u_g$ does
not define a representation of $\esspal$ except in the special case of $E^0$ a trivial bundle over
$\esspal$.) Also we obtain directly from $\alpha$, a representation $\rho_{\esspal} : \esspal
\to ISL(E^0)$ and $\alpha$ also induces a representation $\rho_V : \Bis(\esspal) \to
\mathrm{InnEnd}(E^0)$.

\medskip

The \emph{embedding invariant} $\Phi_{\hookrightarrow}$ of $A$ in $B$ is then fixed by an
approximation of a maximal rank element of $B$, formed by finite linear combinations:-

\begin{equation}
\Phi_{\hookrightarrow} = \{ ~~ \sum_Y ~ p_x u_{ (p_x,p_y) } p_y  ~~ \vert ~~ p_x,p_y \in Y \subset \{
e_{xy},
e_{xy} e_{xy}^*  \} \cong \esspal ~~ \}
 \end{equation}
 for all subsets $Y \in \{ e_{xy}, e_{xy} e_{xy}^* \}$. 
 
 \medskip
 
In general, the probability measure on the pure state space of a
non-commutative algebra is not
unique. However, in our case, since $E^0$ is a
locally trivial Banach bundle, the $p_i$ are all
equivalent, and so there is only one probability measure on $X$. From this, it follows that the faithful
conditional expectation $P : B \to A$ is unique and is given by $P(b) = \int_X \alpha_g(b) dx$ as in
\ref{Examples P}, where $X = \esspal^0$ and where $dx$ is the unique probability measure on $X$ obtained
from the basis of $\H = L^2(E)$, dense in $X$, as defined above. In finite dimensions we have $P(b) =
\sum_i p_i b p_i$. Then $P$ is identified with the unique restriction map $P : C^*(E) \to C^*(E^0)$. 

\medskip

Also, $\cup_{p_i}^n \in A$ forms a unit or an
approximate unit for $B$. \flushright{$\diamond$}
\end{definition}

\section{Acknowledgements}
Many thanks to Paolo Bertozzini, Roberto Conti and Pedro Resende for their helpful
insights (alphabetical order). See title page for affiliation.


\begin{thebibliography}{99999}

\bibitem[AGN]{AGN} 
J. Aastrup, J. Grimstrup, R. Nest,
On spectral triples in quantum gravity I,
\textit{Classical and quantum gravity},
Arxiv: 0802.1783 (2008).

\bibitem[A1]{Arveson}
W. Arveson,
Generators of noncommutative dynamics
Arxiv: 0201137 (2002).

\bibitem[A2]{Arveson paths}
W. Arveson,
Path spaces, continuous tensor products and $E_0$-semigroups,
\textit{Proceedings of the NATO
Advanced Study Institute and Aegean conference,} (1996) Vol. 495 
of NATO ASI Ser. C, Math. Phys. Sci.
Kluwer Academic Publishers Dordrecht 1997 pages 1–111.

\bibitem[Bae]{qq}
J. Baez,
Quantum quandaries: a category-theoretic perspective,
\textit{The Structural Foundations of Quantum Gravity}, D. Rickles, S. French and J. T. Saatsi (Eds),
pages 240–266. Oxford University Press (2004),
Arxiv: quant-ph/0404040v2.

\bibitem[BarCr]{BC model} J.W. Barrett, L. Crane \textit{A Lorentzian signature model for quantum
general relativity}, Class. Quant. Grav. 17 p 3101-3118, (2000).


\bibitem[BeCo]{CDS Roberto}
E. B\'edos, R. Conti,
On discrete twisted C*-dynamical systems, Hilbert
C*-modules and regularity,
\textit{M\"unster J. Math.} 5, 183-208 (2012), 
Arxiv: 1104.1731.

\bibitem[BCL1]{BCL Imp} 
P. Bertozzini, R. Conti, W. Lewkeeratiyutkul, 
A Spectral Theorem for Imprimitivity $C^*$-Bimodules,
Arxiv: 0812.3596 (2008)

\bibitem[BCL2]{BCL En} 
P. Bertozzini, R. Conti R, W. Lewkeeratiyutkul,  (2009),  
Enriched Fell Bundles and Spaceoids,
(``proceedings of the 2010 RIMS thematic year on perspectives on deformation quantization and
noncommutative geometry''),  
ArXiv: 1112.5999

\bibitem[BCL3]{BCL mst}
P. Bertozzini, R. Conti, W. Lewkeeratiyutkul, 
Modular Theory, Non-Commutative Geometry and Quantum Gravity,
\textit{Special Issue ``Noncommutative Spaces and Fields'', SIGMA} 6:067, (2010) 
Arxiv: 1007.4094


\bibitem[BCL4]{BCL Cncg}
P. Bertozzini, R. Conti, W. Lewkeeratiyutkul,  
Categorical non-commutative geometry,
\textit{J. Phys.: Conf. Ser.} 346 012003 (2012).

\bibitem[B]{Rogier}
R. Bos,
Continuous representations of groupoids,
Arxiv: 0612639 (2007).

\bibitem[CPPR]{CPPR}
A. Carey, J. Phillips, I. Putnam, A. Rennie,
Type III KMS states on a class of C*-algebras containing $O_n$ and $\mathcal{Q}_{\mathbb{N}}$ and their
modular index.
\textit{Perspectives on Noncommutative geometry, American Mathematical Society, Fields Institute
Communications} ISBN 978-0-8218-4849-4 (2000).

\bibitem[C1]{essay} A. Connes, Noncommutative geometry and physics,
http://alainconnes.org/docs/einsymp.pdf

\bibitem [C2]{gravity} 
A. Connes, 
Gravity coupled with matter and the foundation of non-commutative geometry, 
\textit{Comm.  Math. Phys.} Vol.182 (1996), N.1, 155-176.


\bibitem[C3]{Connes' book} 
A. Connes, 
\textit{Noncommutative Geometry,}
\textit{Academic Press}, London, (1994).


\bibitem[CC]{sap} 
A. Chamseddine, A. Connes,
The spectral action principle.
\emph{Comm. Math. Phys.} Vol.186 (1997), N.3, 731-750.

\bibitem[CM]{walk}
A. Connes, M. Marcolli,
A walk in the non-commutative garden, (2006),
available at http://www.alainconnes.org/downloads.html

\bibitem[CS]{Roberto}
R. Conti, W. Szymanski,
Submitted to the Proceedings of the EU-NCG 4th Annual Meeting, Bucharest 2011, 
Arxiv: 1108.0860.


\bibitem[Cr1]{qg} L. Crane, 
Categorical geometry and the mathematical foundations of quantum
gravity, 
\textit{Contribution to the Cambridge University Press volume on quantum gravity} (2006).
Arxiv: gr-qc/0602120. 


\bibitem[Cr2]{category qg} 
L. Crane, 
2-d physics and 3-d topology, 
\textit{Commun. Math. Phys.} 135 (1991) 615-640.

\bibitem[Cr3]{wqg}
L. Crane,
What is the mathematical structure of quantum spacetime?
ArXiv:0706.4452 (2007).



\bibitem[D]{Dixmier} 
J. Dixmier  (1982). 
$C^*$-algebras, 
North-Holland Publishing company, English translation.

\bibitem[SD]{Dop}
S. Doplicher,
Spacetime and fields, a quantum texture,
\textit{Proceedings of the 37th Karpacz Winter School of Theoretical Physics}, (2001), 204-213
Arxiv: 0105251 (2003).

\bibitem[E]{non com cartan}
R. Exel,
Noncommutative Cartan sub-algebras of C*-algebras.
Arxiv: 0806.4143 (2008).

\bibitem[FD]{Fell Doran} 
J. Fell J, R. Doran (1998). 
Representations of $C^*$-algebras, Locally Compact Groups and Banach $*$-algebraic bundles,
vol~1-2, Academic Press.

\bibitem[GT]{GT}
K. Giesel, T. Thiemann, 
Algebraic Quantum Gravity (AQG) I. Conceptual Setup,
ArXiv:gr-qc/0607099 (2006).

\bibitem[GLR]{GLR} 
P. Ghez P,  R. Lima, J. Roberts (1985).  
W*-categories,  
\textit{Pacific~J.~Math.} \textbf{120} 11 133-159. 


\bibitem[H]{Harti} 
R. El Harti, 
The structure of a subclass of amenable Banach algebras,
\textit{International Journal of Mathematics and Mathematical Sciences},  (2004).


\bibitem[I]{Isham}
C. Isham,
Topos methods in the foundations of physics,
Arxiv: 1004.3564 (2010).


\bibitem[JL]{trans}
V. Jimenez, G.S. Lopez,
Transitive flows on manifolds,
\textit{Rev. Mat. Iberamericana} 20 (2005), 107-130.

\bibitem[KR1]{KR1}
R. Kadison, J. Ringrose, 
Automorphisms of operator algebras,
\textit{Bulletin of the American Mathematical Society} 72 (1966), no. 6, 1059--1063.

\bibitem[KR2]{KR2}
R. Kadison, J. Ringrose,
Derivations and automorphisms of operator algebras,
\textit{Comm. Math. Phys.} 4, 32-63, (1967).


\bibitem[K1]{fbg} 
A. Kumjian,  
Fell bundles over groupoids, 
Arxiv: math.oa/607230, 
\textit{Proceedings of the American mathematical society}, Vol. 126, No. 4 (Apr., 1998) pp. 1115--1125.

\bibitem[K2]{Kumjian notes}
A. Kumjian,
Notes on C*-algebras of graphs,
Available at:
http://wolfweb.unr.edu/homepage/alex/pub/survey.pdf

\bibitem[K3]{C*diag}
A. Kumjian,
On C*-diagonals,
\textit{Can. J. Math.}, Vol: XXXVIII, No.4, 1986, pp.969-1008.

\bibitem[M1]{sc} 
R. Martins,
 Categorified noncommutative manifolds,
\textit{International Journal of Modern Physics A}, Vol 24, No.15,2802-2819 (2009).
Arxiv: math.ph/0811.1485 

\bibitem[M2]{spectral C}
R. Martins,
Spectral C*-categories and Fell bundles with path-lifting.
Arxiv: 1308.5247 (2013).

\bibitem[MM]{book MM}
I. Moerdijk, J, Mrcun,
Introduction to Foliations and Lie Groupoids,
\textit{Cambridge studies in advanced mathematics, CUP} 2003.

\bibitem[Pa]{Paterson}
A. Paterson,
\textit{Groupoids, inverse semigroups and their operator algebras,}
\textit{Springer} 1999.

\bibitem[P]{fce}
R. Pereira,
Representing conditional expectations as elementary operators.
\textit{Proceedings of the American mathematical society}
Volume 134, Number 1, Pages 253–258 (2005).

\bibitem[Ra]{Raeburn}
I. Raeburn,
Dynamical systems and operator algebras,
http://maths-proceedings.anu.edu.au/CMAProcVol36/CMAProcVol36-Raeburn.pdf

\bibitem[Res]{Pedro}
P. Resende,
\'Etale groupoids and their quantales,
\textit{Adv. Math.} 208 (2007) 147-209,
Arxiv: math/0412478.

\bibitem[Ren]{Renault}
J. Renault,
Cartan subalgebras in C*-algebras
\textit{Irish Math. Soc. Bull.}
61 (2008) 29-65.
Arxiv: 0803.2284

\bibitem[Sc1]{forces} 
T. Sch\"ucker, 
Forces from Connes' geometry, 
ArXiv:hep-th/0111236.
\textit{Lect. Notes Phys}. 659:285-350 (2005).

\bibitem[Sc2]{ncg and sm} 
T. Sch\"ucker, 
Noncommutative geometry and the standard model, 
Arxiv: hep-th/0409077. 
\textit{Int. J. Mod. Phys}. A20:2471-2480 (2005).

\bibitem[St]{Christoph}
C. Stephan,
Almost-commutative geometry, massive neutrinos and the orientability axiom in KO-dimension 6,
Arxiv: hep-th/0610097 (2006).


\bibitem[T]{TomiyamaTensorProducts}
J. Tomiyama,
Tensor products of C*-algebras,
\textit{Publ. RIMS, Kyoto univ. 11 (1995) 163-183}



\end{thebibliography}
\end{document}